\documentclass[11pt,reqno,tbtags]{amsart}

\usepackage{amsthm,amssymb,amsfonts,amsmath}
\usepackage{amscd} 
\usepackage{mathrsfs}
\usepackage[numbers, sort&compress]{natbib}
\usepackage{subcaption, graphicx}
\usepackage{vmargin}
\usepackage{setspace}
\usepackage{paralist}
\usepackage[pagebackref=true]{hyperref}
\usepackage{xcolor}
\usepackage[normalem]{ulem}

\providecommand{\R}{}
\providecommand{\Z}{}
\providecommand{\N}{}

\renewcommand{\R}{\mathbb{R}}
\renewcommand{\Z}{\mathbb{Z}}
\renewcommand{\N}{{\mathbb N}}

\newcommand{\E}[1]{{\mathbf E}\left[#1\right]}										
\newcommand{\e}{{\mathbf E}}

\newcommand{\p}[1]{{\mathbf P}\left\{#1\right\}}
\newcommand{\I}[1]{{\mathbf 1}_{[#1]}}
\newcommand{\set}[1]{\left\{ #1 \right\}}
\newcommand{\Cprob}[2]{\mathbf{P}\set{\left. #1 \; \right| \; #2}} 
\newcommand{\probC}[2]{\mathbf{P}\set{#1 \; \left|  \; #2 \right. }}

\newcommand\cP{\mathcal P}

\newcommand{\pran}[1]{\left(#1\right)}
\providecommand{\eps}{}
\renewcommand{\eps}{\varepsilon}
\DeclareRobustCommand{\SkipTocEntry}[5]{} 
\reversemarginpar
\newtheorem{thm}{Theorem}
\newtheorem{lem}[thm]{Lemma}
\newtheorem{prop}[thm]{Proposition}
\newtheorem{cor}[thm]{Corollary}

\newtheorem{definition}[thm]{Definition}

\newtheorem{fact}[thm]{Fact}
\newtheorem{claim}[thm]{Claim}

\newcommand{\len}{\mathop{\mathrm{len}}}
	
\usepackage[osf]{Baskervaldx} 
\usepackage[baskervaldx,cmintegrals,bigdelims,vvarbb]{newtxmath} 
\usepackage[cal=boondoxo]{mathalfa} 

\newcommand{\low}{(\log n)^{1/5}}

\newcommand{\med}{0.9\,(\log n/\log \log n)}
\newcommand{\nearmax}{0.99\,(\log n/\log \log n)}

\graphicspath{{graphics/}}
\newcommand\urladdrx[1]{{\urladdr{\def~{{\tiny$\sim$}}#1}}} 

\newcommand{\rw}{\mathrm{w}}
\makeatletter
\newcommand{\pushright}[1]{\ifmeasuring@#1\else\omit\hfill$\displaystyle#1$\fi\ignorespaces}
\newcommand{\pushleft}[1]{\ifmeasuring@#1\else\omit$\displaystyle#1$\hfill\fi\ignorespaces}
\makeatother
\newcommand{\dist}{\ensuremath{\mathrm{dist}}}

\newcommand{\rt}{{\bf T}}

\newcommand{\type}{\tau} 
\newcommand{\cw}{\mathrm{c}} 
\usepackage{todonotes}

\newcommand{\mk}{k}
\newcommand{\dyck}{\mathrm{d}}
\numberwithin{thm}{section}
\newcommand{\f}{f}

\begin{document}

\title{The top eigenvalue of uniformly random trees} 
\author{Louigi Addario-Berry}
\address{Department of Mathematics and Statistics, McGill University, 805 Sherbrooke Street West, 
		Montr\'eal, Qu\'ebec, H3A 2K6, Canada}
\email{louigi.addario@mcgill.ca}
\urladdrx{http://problab.ca/louigi/}
\author{G\'abor Lugosi}
\address{Department of Economics and Business, 
Pompeu  Fabra University, Barcelona, Spain;
ICREA, Pg. Lluís Companys 23, 08010 Barcelona, Spain;
Barcelona Graduate School of Economics}
\email{gabor.lugosi@gmail.com}
\urladdrx{http://www.econ.upf.edu/~lugosi/}
\author{Roberto Imbuzeiro Oliveira}
\address{Instituto de Matem\'{a}tica Pura e Aplicada (IMPA), Estrada Dona Castorina 110, Rio de Janeiro, RJ, 22460-320, Brazil}
\email{rimfo@impa.br}
\urladdrx{https://sites.google.com/view/roboliv}
\date{April 2, 2024} 



 \begin{abstract} 
Let ${\mathbf T}_n$ be a uniformly random tree with vertex set $[n]=\{1,\ldots,n\}$, let $\Delta_{{\mathbf T}_n}$ be the largest vertex degree in ${\mathbf T}_n$, and let $\lambda_1({\mathbf T}_n),\ldots,\lambda_n({\mathbf T}_n)$ be the eigenvalues of its adjacency matrix, arranged in decreasing order. We prove that $|\lambda_1({\mathbf T}_n)-\sqrt{\Delta_{{\mathbf T}_n}}| \to 0$ in expectation as $n \to \infty$, and additionally prove probability tail bounds for $|\lambda_1({\mathbf T}_n)-\sqrt{\Delta_{{\mathbf T}_n}}|$. 
Writing $a_n$ for any median of $\Delta_{\rt_n}$, we also prove that $|\lambda_k(\rt_n)-\sqrt{a_n}| \to 0$ in expectation, uniformly over $1 \le k \le e^{\log^\beta(n)}$, for any fixed $\beta \in (0,1/2)$. 

The proof is based on the trace method and thus on counting closed walks in a random tree. To this end, we develop novel combinatorial tools for encoding walks in trees that we expect will find other applications. In order to apply these tools, we show that uniformly random trees --- after appropriate "surgery" --- satisfy, with high probability, the properties required for the combinatorial bounds to be effective. 
\end{abstract}

\maketitle

\section{\bf Introduction}
In this work, we derive precise information about the largest eigenvalue of the adjacency matrix of large uniformly random trees. For a finite graph $G=(V,E)$, the adjacency matrix of $G$ is the $V\times V$ symmetric matrix $A=A(G)$ whose $(u,v)$ entry is $1$ if $\{u,v\} \in E$ and is $0$ otherwise.
Write $\Delta_G$ for the largest degree of any vertex in $G$, and write $\lambda_1(G)\ge \lambda_2(G) \ge \cdots \ge \lambda_n(G)$ for the ordered sequence of eigenvalues of the adjacency matrix of $G$, where $n=|V|$. Also, for positive integers $n$ write $[n]:=\{1,\ldots,n\}$.
\begin{thm}\label{thm:main}
For $n \ge 1$, let $\rt_n$ be a uniformly random tree with vertex set $[n]$. Then 
\[
\e|\lambda_1(\rt_n)-\sqrt{\Delta_{\rt_n}}| \to 0
\]
as $n \to \infty$. Moreover, 
there exists a universal constant $L>0$ such that, as $n \to \infty$, 
\[
\p{|\lambda_1(\rt_n)-\sqrt{\Delta_{\rt_n}}|>L\left(\frac{\log^{2}\Delta_{\rt_n}}{\sqrt{\Delta_{\rt_n}}}\right)^{1/15}} \leq n^{-0.8+o(1)}.
\]
\end{thm}
It is natural to ask whether the terms $\Delta_{\rt_n}$ in the above theorem can be replaced by deterministic values. For the error bars on $|\lambda_1-\sqrt{\Delta_{\rt_n}}|$, they can. It follows from results of \citet[Theorem 2 and Lemma 4]{moon1968} that $\Delta_{\rt_n}\cdot\big(\tfrac{\log n}{\log\log n}\big)^{-1} \to 1$ in probability and that $\p{\Delta_{\rt_n} < 0.99\tfrac{\log n}{\log\log n}}< -\exp(-(1+o(1))n^{0.01})$, and so the probability bound in the above theorem also implies that there is a universal constant $L > 0$ such that 
\[
\p{|\lambda_1(\rt_n)-\sqrt{\Delta_{\rt_n}}|>L\left(\frac{(\log\log n)^{2.5}}{\sqrt{\log n}}\right)^{1/15}} \leq n^{-0.8+o(1)}.
\]
On the other hand, the upper tail of $\Delta_{\rt_n}$ is heavy enough that the probability bound in the theorem would not be true if $|\lambda_1(\rt_n)-\sqrt{\Delta_{\rt_n}}|$ were replaced by, e.g., $\big|\lambda_1(\rt_n)-\sqrt{\tfrac{\log n}{\log\log n}}\big|$, or indeed by any other fixed sequence. However, the expectation bound {\em does} hold with $\sqrt{\Delta_{\rt_n}}$ replaced by its median; moreover, we are able to prove the same expectation bound not just for $\lambda_1(\rt_n)$ but for $\lambda_k(\rt_n)$, for $k$ quite large.
\begin{thm}\label{thm:main2}
For $n \ge 1$, let $\rt_n$ be a uniformly random tree with vertex set $[n]$. Let $a_n$ be any median of $\Delta_{\rt_n}$. Fix any constant $\beta \in (0,1/2)$ and let $k(n)=\lceil \exp((\log n)^\beta)\rceil$. Then 
 \[
 \e|\lambda_1(\rt_n)-\sqrt{a_n}| \to 0
 \]
and
\[
 \e|\lambda_{k(n)}(\rt_n)-\sqrt{a_n}| \to 0. 
 \]
\end{thm}
Together with Theorem~\ref{thm:main}, it follows that also 
\[
 \e|\lambda_{k(n)}(\rt_n)-\sqrt{\Delta_{\rt_n}}| \to 0, 
\]
by the triangle inequality. Moreover, 
since $\lambda_1 \ge \lambda _i \ge \lambda_{k(n)}$ for all $i \in [k(n)]$, the same estimates hold for $\e|\lambda_{i(n)}(\rt_n)-\sqrt{a_n}|$ and for $\e|\lambda_{i(n)}(\rt_n)-\sqrt{\Delta_{\rt_n}}|$, for any sequence $(i(n),n \ge 1)$ with $1\le i(n) \le k(n)$. We conjecture that for any $\beta \in (1/2,1]$ there is $c=c(\beta) > 0$ such that, with $\ell(n)=\lceil \exp((\log n)^\beta)\rceil$, then $\p{\lambda_{\ell(n)}(\mathrm{T}_n)\le \sqrt{a_n}-c} \to 1$; this would imply that the range of $k(n)$ covered by Theorem~\ref{thm:main2} is essentially optimal.

We expect that the methodology we develop here can also be used for controlling the top eigenvalues of other random tree models. Informally, in order to be applicable, it requires maximum-degree vertices to be relatively few in number and to be well-separated from one-another in the tree. (We provide a fairly detailed overview of our proof technique in Section~\ref{sub:overview}.) However, for the moment we do not see a way to formulate a more precise set of requirements which would allow us to treat many models simultaneously; thus far, model-specific conditions and calculations seem unavoidable to us.

\subsection{Related work}

It is known in some generality that the bulk spectral properties of sparse random graphs are determined locally. In particular, suppose that $(G_n,n \ge 1)$ is a sequence of finite random rooted graphs, where $G_n$ has $n$ vertices, and write 
\[
\mu_n = \frac{1}{n}
\sum_{i=1}^n \delta_{\lambda_i(G_n)}
\]
for the empirical spectral measure of $G_n$. If the sequence $(G_n,n \ge 1)$
converges in distribution in the local weak sense to a locally finite random rooted graph $G$, then $\mu_n$ converges weakly to a limiting measure $\mu=\mu_G$; see \cite{MR2724665}, where this is proved under a uniform integrability condition on the root degree, and \cite{bordenave16notes,abert16pointwise}, which handle the general case. More precise results about the bulk of the spectrum are available when the local weak limit is a (deterministic or random) tree \cite{MR2956206,MR4260473,MR4645723,MR4049219}.

For the edge of the spectrum, fewer general results are available, and indeed one expects more model-dependence. In the setting of the Erd\H{o}s-R\'enyi random graph $G(n,p)$, Krivelevich and Frieze showed that for any sequence $(p_n,n \ge 1)$ of probabilities, $\lambda_1(G(n,p_n))=(1+o(1))\max(np_n,\sqrt{\Delta_{G(n,p_n)}})$ in probability as $n \to \infty$, where $\Delta_{G(n,p_n)}$ denotes the maximum degree of the graph. The threshold $np_n \approx \sqrt{\Delta_{G(n,p_n)}}$ between the regime where $np_n$ is larger and that where $\sqrt{\Delta_{G(n,p_n)}})$ is larger occurs around $p_n^* =  (\log n)/\log(4/e)$. Alt, Ducatez and Knowles \cite{adk21,MR4515695,MR4328063} have shown that $p_n^*$ is the boundary between the localized and delocalized phases for the spectral edge of sparse Erd\H{o}s-R\'enyi random graphs, in the sense that for $\log\log n \ll p_n<(1-o(1))p_n^*$ the eigenvectors corresponding to eigenvalues near the edge of the spectrum are concentrated on a small number of vertices, whereas for $p_n > (1+o(1))p_n^*$ this is not the case. (We are being slightly informal here, and refer the reader to the papers cited above for more precise statements.) To see how the value $\sqrt{\Delta_{G(n,p_n)}}$ arises, note that for any finite graph $G=(V,E)$, if $v$ is a vertex of $G$ with degree $\Delta_G$, then the vector 
$x:V \to \R$ given by 
\[
x_w = \begin{cases}
        \sqrt{\Delta_G}& \mbox{ if }w=v\\
        1&\mbox{ if }\{v,w\} \in E\\
        0&\mbox{ otherwise}
        \end{cases}
\]
satisfies that $\langle x,Ax\rangle \ge \sqrt{\Delta_G}|x|^2$, so by Rayleigh's formula $x$ is a witness that $\lambda_1(G) \ge \sqrt{\Delta_G}$. Moreover, $x$ is very localized, taking nonzero values only on a single vertex and its neighbourhood; so this suggests that if  $\lambda_1(G)\approx\sqrt{\Delta_G}$ then perhaps the extremal eigenvectors are also localized.

Very recently, Hiesmayr and McKenzie \cite{hm23} have built on the techniques of \cite{MR4515695} to describe the extremal eigenvalues and eigenvectors of $G_{n,p}$ in the regime $\log^{-1/15}n \le np \le \log^{1/40} n$, showing that  $|\lambda_1(G_{n,p})-\sqrt{\Delta_{G_{n,p}}}| \to 0$ in probability, and, more strongly, that with high probability each of the top $e^{\log^{1/8} n}$ eivenvalues of $G_{n,p}$ is essentially determined by the graph structure of the metric ball of radius two around a vertex of near-maximal degree. Hiesmayr and McKenzie in fact prove quantitative bounds, and also establish localization of the corresponding eigenvectors. The case $np=\Theta(1)$ of their result in particular answers a question of Guionnet \cite[Section 3.3]{guionnet21bernoulli}. 

The above results for the spectral edge of Erd\H{o}s-R\'enyi random graphs have some similarity to ours; and indeed, the local structure of an Erd\H{o}s-R\'enyi random graph is similar to that of the random trees whose study we undertake (both are locally described by a Poisson branching process). We expect that eigenvector localization near the spectral edge also occurs in many random tree models, including those considered in this paper, but our proof technique - which we now sketch - only gives access to the extremal eigenvalues, and not to the associated eigenvectors.

\subsection{Proof overview}\label{sub:overview}

This section both presents our approach to proving Theorems~\ref{thm:main} and~\ref{thm:main2} and introduces some concepts and several pieces of notation which are required in the sequel. 

The bulk of the work is in proving the upper bound of the theorem. 
Our proof of the upper bound is based on the trace method, which we now recall. 
Fix a graph $G=(V,E)$ with $|V|=n$ and write $A=A(G)$. 
Then $\{\lambda_i(A^{\mk}),i \in [n]\}=\{\lambda_i(A)^{\mk},i \in [n]\}$ for all ${\mk} \ge 1$, so 
\begin{align}
\lambda_1(G) & 
= \lim_{{\mk} \to \infty} \Big(\sum_{i \in [n]} \lambda_i(A)^{2{\mk}}\Big)^{1/(2{\mk})}\nonumber\\ & 
= \lim_{{\mk} \to \infty} \Big(\sum_{i \in [n]} \lambda_i(A^{2{\mk}})\Big)^{1/(2{\mk})} \nonumber\\ & 
= \lim_{{\mk} \to \infty} \Big(\sum_{v \in V} (A^{2{\mk}})_{vv}\Big)^{1/(2{\mk})} \nonumber\\ &
= \lim_{{\mk} \to \infty} \Big(\max_{v \in V} (A^{2{\mk}})_{vv}\Big)^{1/(2{\mk})}\nonumber\, ,
\end{align}
where for the third equality we have used the trace formula.
Now, $(A^{2{\mk}})_{vv}$ is nothing but the number of closed walks of length $2{\mk}$ from $v$ to $v$ in $G$, by this we mean a sequence of vertices $(w_0,w_1,\ldots,w_{2{\mk}})$ with $w_0=w_{2{\mk}}=v_i$ such that $\{w_{i-1},w_{i}\} \in E$ for each $i\in [2{\mk}]$. The preceding displayed identity thus yields an equivalence between estimating $\lambda_1(G)$ and estimating the number of long closed walks in $G$. Writing $\mathcal{N}_{\mk}(v,G)$ for the set of closed walks of length $2{\mk}$ from $v$ to $v$ in $G$ and $N_{\mk}(v,G)=|\mathcal{N}_{\mk}(v,G)|$ and $N_{\mk}(G)=\max_{v \in V} N_{\mk}(v,G)$, we may rewrite the above identity for $\lambda_1(G)$ as 
\begin{equation}\label{eq:lambdaonebound}
    \lambda_1(G)=\lim_{{\mk} \to \infty} \left(\max_{v \in V} N_{\mk}(v,G)\right)^{1/(2{\mk})}
    =\lim_{{\mk} \to \infty} N_{\mk}(G)^{1/(2{\mk})}
    \, .
\end{equation}

We now turn to the setting of trees. For a tree $T=(V,E)$ and a walk $\rw=(w_0,w_1,\ldots,w_{j})$ in $T$, we write $D(\rw)=(d_0(\rw),\ldots,d_{j}(\rw)):=(\dist_T(w_0,w_i),0 \le i \le j)$ for the sequence of graph distances to $w_0$ in $T$, along the walk $\rw$. If $\rw$ is a closed walk then $j=2{\mk}$ is even and  $D(\rw)$ is a {\em Dyck path} 
 --- a non-negative lattice path with $d_0(\rw)=d_{2{\mk}}(\rw)=0$ and $|d_i(\rw)-d_{i-1}(\rw)|=1$ for all $i \in [2{\mk}]$. We will usually be considering $D(\rw)$ for closed walks $\rw$, and have chosen the notation accordingly.
We may then write 
\[
N_{\mk}(v,T)
=\sum_{\dyck} 
\{\rw \in \mathcal{N}_{\mk}(v,T): D(\rw)=\dyck\}\, ,
\]
where the sum is over Dyck paths $\dyck$ of length $2{\mk}$.

An easy way to bound this expression from above is to note that, whatever $\dyck$ may be, 
\[
\{\rw \in \mathcal{N}_{\mk}(v,T): D(\rw)=\dyck\}\le \Delta_T^{\mk}\, .
\]
This holds since in any Dyck path $D$ of length $2{\mk}$ there are $m$ increasing steps. An increasing step in $D(\rw)$ corresponds to a step of $\rw$ which increases the distance to $v$ in $T$; and, from any vertex $u$ of $T$ there are at most $\Delta_T$ choices for such a step. The decreasing steps yield no additional choices, since $T$ is a tree; every vertex $u$ of $T$ aside from $v$ has a unique neighbour $u'$ with $\dist_T(u',v) = \dist_T(u,v)-1$. 
Since the number of Dyck paths of length $2{\mk}$ is $\tfrac{1}{{\mk}+1}\binom{2{\mk}}{{\mk}}$, we deduce that 
\begin{equation}\label{eq:trivialbound}
N_{\mk}(T)=\max_{v \in V} N_{\mk}(v,T) \le \frac{1}{{\mk}+1}\binom{2{\mk}}{{\mk}}  \Delta^{\mk} \le (4\Delta)^{{\mk}}.
\end{equation}
Using \eqref{eq:lambdaonebound}, this yields the well-known fact that 
$\lambda_1(T) \le \lim_{{\mk} \to \infty} (\max_v N_{\mk}(v,T))^{1/2{\mk}}\le 2\sqrt{\Delta_T}$. 
This bound is tight for general trees for large $\Delta_T$; see, e.g.,
Godsil \cite{God84}, Stevanovi{\'c} \cite{Ste03}  who show that for any tree, the largest eigenvalue is at most $2\sqrt{\Delta_T-1}$.\footnote{In fact, we may easily recover the bound $2\sqrt{\Delta_T-1}$ from this argument. To do so, simply note that in \eqref{eq:lambdaonebound} we can replace $\max_{v \in V} N_{\mk}(v,T)$ by $N_{\mk}(w,T)$ for {\em any} fixed vertex $w$. If we fix a vertex $w$ of degree less than $\Delta_T$, then the number of {\em increasing} steps from any vertex is at most $\Delta_T-1$.}

We next illustrate the basic idea of how we  improve this bound under additional conditions on the tree, but first introduce one required piece of notation: for a tree $T=(V,E)$ and $r \in \N$, let $\Delta_T^{(r)}=\max_{v \in V} |\{w \in V:\dist_T(v,w)=r\}|$ be the largest size of an $r$'th neighbourhood in~$T$.

Note that any Dyck path $\dyck=d_0d_1\ldots d_{2{\mk}}$ either begins with the string $010$ or with the string $012$. (We shall often write $d_0d_1\ldots d_{2{\mk}}$ instead of $(d_0,\ldots,d_{2{\mk}})$ and $010$ instead of $(0,1,0)$, {\em et cetera}, for succinctness.) It is straightforward to see that 
\[
|\{\rw \in \mathcal{N}_m(v,T): 
d_0(\rw)d_1(\rw)d_2(\rw)=010\}| \le \Delta_T N_{{\mk}-1}(v,T) \le \Delta_T N_{{\mk}-1}(T),
\]
since a walk counted in the above set is simply the concatenation of a walk of length $2$ and a walk of length $2({\mk}-1)$.
A walk $\rw \in \mathcal{N}_{\mk}(v,T)$ with $d_0(\rw)d_1(\rw)d_2(\rw)=012$ may also be decomposed: after $w_2$ must appear a closed walk from $w_2$ to $w_2$ which avoids $w_1$, followed by a closed walk from $w_1$ to $w_1$ which avoids $w_0$, and finally a closed walk from $w_0$ to $w_0$; each of these may be a walk of length zero. Since there are at most $\Delta_T^{(2)}$ choices for $w_2$, it follows that 
\begin{align*}
|\{\rw \in \mathcal{N}_{\mk}(v,T): d_0(\rw)d_1(\rw)d_2(\rw)=012\}|
\le
\Delta^{(2)}_T\cdot
\mathop{\sum_{i,j,k \ge 0}}_{i+j+k={\mk}-2}N_i(T)N_j(T)N_k(T)\, .
\end{align*}
Combining these bounds, we obtain that 
\[
N_{\mk}(T)\le 
\Delta_T N_{{\mk}-1}(T)
+\Delta^{(2)}_T\cdot
\mathop{\sum_{i,j,k \ge 0}}_{i+j+k={\mk}-2}N_i(T)N_j(T)N_k(T)
\]
We may then apply this case analysis recursively to bound each of the terms $N_{{\mk}-1}(T)$, $N_i(T)$, $N_j(T)$ and $N_k(T)$, and hope that something useful comes out of it.

The set of words $\{010,012\}$ is a {\em prefix code} for Dyck paths; any Dyck path has a unique prefix in this set. For example, another prefix code is $\{010,0121,0123\}$. 
We may generalize the above case analysis and recursive approach to any set of prefix codes. Our bounds end up depending crucially on two quantities. First, we require control on the greatest number of ways that each given code word can appear as 
a prefix of a Dyck path of a walk in the graph (for example, our bound on this for the word 012 was $\Delta^{(2)}_T$). Second, when controlling
\[
|\{\rw \in \mathcal{N}_{\mk}(v,T): D(\rw)=\dyck\}|
\]
for a given Dyck path $\dyck$, our bounds depend on the number of times each code word appears when recursively decomposing $\dyck$ until no further decomposition is possible. 

In order to explain more precisely, a little combinatorics is unavoidable. Say that a sequence $\cw=(c_i,0 \le i \le j)$ is a {\em meander} if $c_0=0$, $c_i \ge 0$ for all $0 \le i \le j$, and $|c_i-c_{i-1}|=1$ for all $i \in [j]$; it is an {\em excursion} if additionally $c_j=0$. The {\em final value} of $\cw$ is $\f(\cw):=c_j$; its {\em length} is $\len(\cw)=j$. 
A finite set $C$ of meanders is a {\em prefix code} (or just {\em code}, for short)
if for any Dyck path $\dyck$ of positive length there is a unique meander $\cw \in C$ such that $\cw$ is a prefix of $\dyck$. The unique code $C$ with $01 \in C$ is $C=\{01\}$; we call this code {\em trivial}. Observe that if $C$ is nontrivial then $010 \in C$. 

For a prefix code $C=(\cw^{(1)},\ldots,\cw^{(a)})$, to any Dyck path $\dyck=(d_0,d_1,\ldots,d_{2{\mk}})$ we associate the vector 
\[
\vec{t}(C,\dyck)=(t_b(C,\dyck),1 \le b \le a),
\]
where $t_b(C,\dyck)$ is the number of times the code word $\cw^{(b)}$ is used when decomposing $\dyck$ using the code $C$. Formally, let $p=p(\dyck,\cw) \in [a]$ be the unique index for which $\cw^{(p)}$ is a prefix of $\dyck$. Let  $\eta_0,\ldots,\eta_{\f(\cw^{(p)})+1}$ be the unique increasing sequence of integers with $\eta_0=\len(\cw^{(p)})$ and $\eta_{\f(\cw^{(p)})+1}=2{\mk}$ such that for all $0 \le j \le \f(\cw^{(p)})$,
\[
\dyck^{(j)}:=(d_{\eta_j+i}-(\f(\cw^{(p)})-j),0\le i \le \eta_{j+1}-\eta_j)
\]
is a Dyck path. Then for $1 \le b \le a$ inductively define 
\[
t_{b}(C,\dyck)
=\I{b=p} +\sum_{j=0}^{\f(\cw_p)} t_{b}(C,\dyck^{(j)})\, .
\]
Note that the length of $\dyck$ is the sum of $\len(\cw^{(p)}),\f(\cw^{(p)})$, and the lengths of 
$\dyck^{(0)},\ldots,\dyck^{(\f(\cw^{(p)}))}$, so by induction we have  
\[
2\mk=\sum_{b \in [a]} t_b(C,\dyck)(\len(\cw^{(b)})+\f(\cw^{(b)})). 
\]
With a slight abuse of notation, given a closed walk $\rw$ in a tree $T$ we also define 
\[
\vec{t}(C,\rw)=\vec{t}(C,D(\rw))\, . 
\]
For a meander $\cw$ of length $j$ we let 
\[
\Delta_T(\cw) = \max_{v \in V}|\{\mbox{walks $\rw=(w_0,\ldots,w_j)$ in }T: w_0=v, D(w)=\cw\}|\,.
\]
For example,  
if $\cw=012\ldots h$ then $\Delta_T(\cw)=\Delta_T^{(h)}$; for another example, if $\cw=01212$ then $\Delta_T(\cw) \le \Delta_T^{(2)}\cdot (\Delta_T-1)$, since for any $v \in V$, a walk  $\rw=w_0w_1w_2w_3w_4$ with $w_0=v$ and with $D(\rw)=01212$ is uniquely determined by choosing $w_2$ with $\dist_T(w_0,w_2)=2$ and  $w_4 \ne w_0$ with $\dist_T(w_1,w_4)=1$. 

The following proposition gives an idea of how decomposing walks by prefix codes allows us to bound the number of closed walks in a tree. 
Given a prefix code $C=(\cw^{(1)},\ldots,\cw^{(a)})$ and a vector $\vec{t}=(t_1,\ldots,t_a)$ of non-negative integers,  
for a vertex $v$ in a tree $T$ we  denote by $N_{C,\vec{t}}(v)$ the number of closed walks $\rw$ from $v$ to $v$ in $T$ with $t_b(C,\rw)=t_b$ for all $1\le b \le a$; the tree $T$ in question will always be clear from context. Note that the length of any such walk is $2\mk=2\mk(\vec{t}):=\sum_{b=1}^a t_b(\len(\cw^{(b)})+\f(\cw^{(b)}))$. 

\begin{prop}\label{prop:kl_bd} Fix a tree $T$, a code $C=\{\cw^{(b)},1 \le b \le a\}$ and a vector of non-negative integers $\vec{t}=(t_1,\ldots,t_a)$. Let
$q= q(\vec{t})=\sum_{b=1}^a t_b\cdot (1+\f(\cw^{(b)}))$
and $\mk=\mk(\vec{t})=\sum_{b=1}^a t_b(\len(\cw^{(b)})+\f(\cw^{(b)}))/2$.
Then for any $\Delta > 0$, setting 
\[
g_b = \Big((1\vee 2e\f(\cw^{(b)})))\cdot \frac{\Delta_T(\cw^{(b)})}{\Delta^{(\len(\cw^{(b)})+\f(\cw^{(b)}))/2}}\Big)^{1/(1+\f(\cw^{(b)}))}\, ,
\]
for any vertex $v$ of $T$ we have 
\[
N_{C,\vec{t}}(v)\le eq^{1/2} \Delta^\mk \Big(\sum_{b=1}^a g_b\Big)^{q}\, .
\]
\end{prop}
The proposition is proved by a careful analysis of the number of ways that a Dyck path can be partitioned by a given code, together with an application of a convexity bound (specifically, the non-negativity of the Kullback-Liebler divergence).

To get a feeling for how one can exploit such a bound, we consider what it yields for a couple of simple codes. 
First, consider the trivial code $C^*=\{01\}$. 
Then $\len(\cw^{(1)})=1$, $\f^{(1)}:=\f(\cw^{(1)})=1$, and therefore if $\len(\rw)=2\mk$ then $q(\vec{t})=2\mk=2t_1$. 
In this case the vector $\vec{t}$ has only one component $t_1$ and the only possible value 
is $\vec{t}=(\mk)$. Thus,
the number $N_{\mk}(v,T)$ of closed walks from $v$ to $v$ in $T$ of length $2\mk$ equals $N_{C,\vec{t}}(v)$ and $q(\vec{t})=2\mk(\vec{t})$. Taking $\Delta=\Delta_T$, then $g_1=2e$, and hence 
the bound of Proposition~\ref{prop:kl_bd} becomes $N_{\mk}(v,T) \le e\mk^{1/2}(4e^2\Delta_T)^\mk$.

Next consider the code
$C^*=\{\cw^{(1)},\cw^{(2)},\cw^{(3)}\}=\{010,0121,0123\}$.
For this code we have
$\f^{(1)}=0,\f^{(2)}=1,\f^{(3)}=3$ and 
$\len(\cw^{(1)})=2, \len(\cw^{(2)}) =3, \len(\cw^{(3)})=3$. For any closed walk $\rw=(w_0,\ldots,w_{2\mk})$ in $T$, writing $t_i=t_i(C^*,\rw)$, we have 
\[
2\mk = 2t_1+4t_2+6t_3,  
\]
and $q=q(\vec{t}(C^*,\rw)) = t_1+2t_2+4t_3 = \mk+t_3$. 
For any tree $T$ we have $\Delta_T(\cw^{(b)})  = \Delta_T^{(b)}$ for $b=1,2,3$, since in each case an embedding of a code word $\cw^{(b)}$ is uniquely specified by the choice of its initial vertex and the unique vertex at distance $b$. We thus have $g_1 = 1$, $g_2 =(2e\Delta^{(2)}/\Delta_T^2)$, and $g_3 = (6e\Delta^{(3)}/\Delta_T^3)^{1/4}$; it follows 
by Proposition~\ref{prop:kl_bd}, again applied with $\Delta=\Delta_T$ that for any vertex~$v$, 
\[
N_{C,\vec{t}}(v) \le eq^{1/2} \Delta_T^\mk
\left(1+ \sqrt{\frac{2e\Delta_T^{(2)}}{\Delta_T^2}} +
      \sqrt[4]{\frac{6e\Delta_T^{(3)}}{\Delta_T^3}} \right)^{\!q}
\le e(2\mk)^{1/2} \Delta_T^\mk
\left(1+ \sqrt{\frac{2e\Delta_T^{(2)}}{\Delta_T^2}} +
      \sqrt[4]{\frac{6e\Delta_T^{(3)}}{\Delta_T^3}} \right)^{\!2\mk},
\]
where we have used that $q=\mk+t_3 \le 2\mk$. 
Thus, the number $N_{\mk}(v,T)$ of closed walks from $v$ to $v$ in $T$ of length $2\mk$ may be bounded as 
\begin{align*} 
N_{\mk}(v,T)  &\le 
e(2\mk)^{1/2} \Delta_T^\mk \cdot \sum_{t_1+2t_2+3t_3=\mk} 
\left(1+ \sqrt{\frac{2e\Delta_T^{(2)}}{\Delta_T^2}} + 
      \sqrt[4]{\frac{6e\Delta_T^{(3)}}{\Delta_T^3}} \right)^{2\mk} \\ 
      & \le e(2\mk)^{1/2}\Delta_T^\mk \mk^{3} \left(1+ \sqrt{\frac{2e\Delta_T^{(2)}}{\Delta_T^2}} + 
      \sqrt[4]{\frac{6e\Delta_T^{(3)}}{\Delta_T^3}} \right)^{2\mk}\, . 
\end{align*} 
Taking $2\mk$'th roots on both sides, letting $\mk \to \infty$  and applying (\ref{eq:lambdaonebound}) this yields that 
\begin{equation}\label{eq:simplebound}
\lambda_1(T) \le \Delta_T^{1/2}\left(1+ \sqrt{\frac{2e\Delta_T^{(2)}}{\Delta_T^2}} + 
      \sqrt[4]{\frac{6e\Delta_T^{(3)}}{\Delta_T^3}} \right)\, . 
  \end{equation}
  This bound is useful if one can control the ratios $\Delta_T^{(2)}/\Delta_T^2$ and $\Delta_T^{(3)}/\Delta_T^3$
 in a meaningful manner, which is the case for the random trees considered in this paper.
  Indeed, it follows from the structural properties shown later\footnote{Specifically, the upper bound on third neighbourhood sizes from  Corollary~\ref{cor:typicalclusterdegrees}, together with the lower bound on $\Delta_{\rt_n}$ coming from Lemma~\ref{lem:typical}.}
  that, with high probability, for all vertices $v$ in the random tree $\rt_n$, the number of vertices within distance three from 
  $v$ is $O(\Delta_{\rt_n}\log^3\Delta_{\rt_n})$, and therefore  $(\Delta_{\rt_n}^{(2)}/\Delta_{\rt_n}^2)^{1/2}=O\left(\Delta_{\rt_n}^{-1/2}\log^{3/2} \Delta_{\rt_n}\right)$
  and $(\Delta_{\rt_n}^{(3)}/\Delta_{\rt_n}^3)^{1/4}=O\left(\log^{1/4}\Delta_{\rt_n}/\Delta_{\rt_n}^{1/2}\right)$.
  This implies that, with high probability, the largest eigenvalue of $\rt_n$ is at most
\[
    \sqrt{\Delta_{\rt_n}} + O\left(\log^{3/2} \Delta_{\rt_n}\right)~.
\]

Unfortunately, Proposition~\ref{prop:kl_bd} is not powerful enough to prove Theorem~\ref{thm:main}. 
Actually, in order to get strong enough bounds on the number of long closed walks in $\rt_n$, we end up additionally needing to keep track of the degrees of vertices we visit along the walk. This makes our codes more complicated, since they need to also encode this information. We partition the nodes of a tree $T$ into three types: high-degree nodes, with degree at least $0.95\Delta_T$; medium-degree nodes, with degree between $\Delta_T^{1/5}$ and $0.95\Delta_T$; and low-degree nodes, which are all the rest. We then augment the Dyck path encoding the distance of a closed path from its starting location with a sequence of vertex types along the sequence. For example, an augmented Dyck path may read
\[
((0,h),(1,h),(2,l),(1,h),(2,m),(1,h),(0,h)),
\]
corresponding to a walk which moved from a high-degree node $v$ to a neighbouring high-degree node $w$, then made two ``zig-zag'' steps, visiting a low-degree and a medium-degree neighbour of $w$ in turn, each time returning to $w$, and finally returned to $v$. 

Recall that a star tree
with central vertex of degree $\Delta$ has maximum eigenvalue $\sqrt{\Delta}$, and the only possible closed walks on such a star (starting from the central vertex) consist of a sequence of ``zigzag'' steps from a high-degree vertex to low-degree vertices and back again. 
Given that we aim to prove that $\e|\lambda_1(\rt_n)-\sqrt{\Delta_{\rt_n}}|=o(1)$, this suggests that for large $\mk$ the dominant contribution to $N_{\mk}(v,\rt_n)$ should consist of zigzag steps from a high-degree vertex to its low or medium-degree neighbours. It turns out to be useful to bound the contribution coming from such steps separately. Thus, we prove bounds on $N_{\mk}(v,T)$ by first bounding the number of walks of length $2j \le 2{\mk}$ which make no ``zigzags from high-degree vertices'', then estimating how many ways such walks can be decorated with zigzags to form a walk of length $2{\mk}$. (For technical reasons we also need to initially count walks which make no ``zigzags of height two'', of the form $((0,m/h),(1,l),(2,l),(1,l),(0,m/h))$, and then also control the number of ways such height-two zigzags can be added.) The bound we end up using appears as Proposition~\ref{prop:countpathsIJ}, in Section~\ref{sec:refined}. We apply it with a code consisting of 50 words, which appears in Figure~\ref{fig:rightcode} in Section~\ref{sec:endofproof}; the need to track vertex types as well as distance is the reason we need so many code words.

It turns out that to apply Proposition~\ref{prop:countpathsIJ} to the tree $\rt_n$ in order to bound $\e|\lambda_1(\rt_n)-\sqrt{\Delta_{\rt_n}}|$, it would suffice to show that $\rt_n$ satisfied the following set of structural properties with sufficiently high probability. 
For a tree $T=(V,E)$ and $v \in V$, write $\deg_T(v)$ for the degree of $v$ in $T$, and write $C_T(v)$ for the set of all vertices $w\in V$ such that the unique path from $v$ to $w$ does not contain two consecutive vertices with degree $<\low$; we call $C_T(v)$ the cluster of $v$ in $T$. 
Then the properties we end up requiring of a tree $T=(V,E)$ with $n$ vertices, in order for our combinatorial bounds to be strong enough, are the following. 
\begin{enumerate}
\item[{\bf N1}] The maximum degree of $T$ is $\Delta_T\geq \nearmax$. 
\item[{\bf N2}] For all $x\in V$, $\sum_{y\in C_T(v)}\deg_{T}(y)\I{\deg_{T}(y) \ge \low}\leq 3 \log n$.
\item[{\bf N3}] There are no adjacent nodes that both have degree $\geq 0.9\log n/\log \log n$.
\item[{\bf N4}] For all $x\in V$, $C_{T}(x)$ contains $<9\,\log \log n$ vertices of degree $\geq \low$.
\end{enumerate}

We prove that a random tree $\rt_n$ satisfies ${\bf N1-N3}$ with high probability by instead analyzing the probability that these properties hold for the family tree of a Poisson$(1)$ branching process conditioned to have total size $n$, then using the fact that the graph structure of $\rt_n$ is equivalent to that of such a conditioned branching process. 
However, we were unable to thereby establish that $\rt_n$ satisfies  property ${\bf N4}$ with high probability, and we do not in fact believe it to be true. What saves the proof, and allows us to make the above approach work, is a {\em rewiring lemma}, which permits us to ``clean up'' our tree to obtain one with the requisite properties, without decreasing $\lambda_1$. This rewiring lemma seems like a potentially valuable general tool, and its proof is quite simple, so we close the sketch of the upper bound by stating it and immediately providing its proof.

Given a finite tree $T=(V,E)$ 
and distinct vertices $v,w\in V$, the $(v,w)$-rewiring of $T$, denoted by $T_{v,w}$, is the graph obtained from $T$ as follows: let $\Gamma(v,T)$ denote the neighborhood of $v$ in $T$, and $u\in \Gamma(v,T)$ be the unique neighbor of $v$ on the path connecting $v$ to $w$ in $T$. Then the rewired tree $T_{v,w}=(V,E_{v,w})$ has the same vertex set as $T$ and edge set 
\[E_{v,w}:=E\cup \{sw\,:\,s\in \Gamma(v,T)-\{u\}\} - \{sv\,:\,s\in \Gamma(v,T)-\{u\}\}.\]
In words, $T_{v,w}$ is obtained from $T$ by replacing each edge $sv$ such that $s\neq u$ by the edge $sw$. 
 The following fact is immediate.
 \begin{fact}\label{fact:rewiring}$T_{v,w}$ is still a tree, and the unique path connecting $v$ to $w$ in $T_{v,w}$ is the same as in $T$. Moreover, the degrees $\mathrm{deg}_{T_{v,w}}(x)$, for $x\in V$, satisfy:
 \begin{equation}\label{eq:sendedges}\mathrm{deg}_{T_{v,w}}(x)=\left\{\begin{array}{ll}1, & x=v;\\ \mathrm{deg}_T(v) + \mathrm{deg}_T(w) - 1, & x=w;\\ \mathrm{deg}_T(x), & x\in V\backslash \{v,w\}.\end{array}\right.\end{equation}\end{fact}
\begin{lem}[Rewiring lemma]\label{lem:rewiring}With the above notation, 
$\max\{\lambda_1(T_{v,w}),\lambda_1(T_{w,v})\}\geq \lambda_1(T).$\end{lem}
\begin{proof}Assume $V=[n]$ for simplicity and let $A$ denote the adjacency matrix of $T$. By Perron-Frobenius, its largest eigenvalue is achieved by a vector $x\in\R^n$ with $|x|=1$ and non-negative entries, so that:
\[\langle x, Ax\rangle = 2\,\sum_{\{i,j\}\in E}x_ix_j = \lambda_1(T)\geq 0.\]
Assume without loss of generality that $x_v\leq x_w$, and let $B$ denote the adjacency matrix of $T_{v,w}$. We will show that $\lambda_1(T_{v,w})\geq \lambda_1(T)$. In fact, since \[\langle x,Bx\rangle \leq \lambda_1(T_{v,w})|x|^2 = \lambda_1(T_{v,w}),\] it suffices to argue that $\langle x, Bx\rangle\geq \langle x, Ax\rangle$. To do this, let $u\in \Gamma_T(v)$ be as in the definition of $T_{v,w}$. Then 
\[\langle x, Bx\rangle -\langle x, Ax\rangle = 2\,\sum_{s\in \Gamma_T(v)\backslash\{u\}}(x_w-x_v)\,x_s\geq 0,\]
because $x_s\geq 0$ for all $s\in [n]$ and $x_w\geq x_v$ by assumption.\end{proof}

We conclude the overview by briefly sketching the proof of Theorem~\ref{thm:main2}, which is rather straightforward once Theorem~\ref{thm:main} is at hand. Since $\lambda_k \le \lambda_1$ for all $k \ge 1$, to prove the theorem it suffices to prove that $\e{\lambda_1(\rt_n)} \le \sqrt{a_n}+o(1)$ and that $\e{\lambda_k(\rt_n)} \ge \sqrt{a_n}-o(1)$, for $a_n$ as in the theorem. We accomplish this via a straightforward but useful fact about trees: if $T$ is a rooted tree, and $T$ has $2s$ nodes with at least $c$ children, then $\lambda_s(T) \ge \sqrt{c}$. This is proved by first constructing a set $(x^{(1)},\ldots,x^{(s)})$ of orthogonal vectors such that $\langle x^{(i)},Ax^{(i)}\rangle \ge \sqrt{c}|x^{(i)}|^2$ for $1 \le i \le s$, then using the Rayleigh quotient principle (see Lemma~\ref{lem:tree_ev_bd}, below). The lemma allows us to bound $\e\lambda_k(\rt_n)$ from below by showing that with high probability $\rt_n$ contains many vertices of near-maximal degree, which we accomplish using the fact that the graph structure of $\rt_n$ is the same as that of a Poisson branching process conditioned to have total progeny $n$. The latter fact also allows us to  show that $\e\sqrt{\Delta_n}$ is within $o(1)$ of $\sqrt{a_n}$, which, together with Theorem~\ref{thm:main}, yields the required upper bound on $\e{\lambda_1(\rt_n)}$.

\label{sec:overview}
\subsection{Overview of the remainder of the paper}

The rest of the paper is structured as follows. In Section \ref{sec:combinatorialnew} we establish the combinatorial tools that we use for counting walks in trees, culminating in Corollary \ref{cor:countpathsIJ}
that gives a general upper bound in terms of certain properties of the tree and the code used for encoding walks. 
In Section \ref{sec:structural} we prove some structural properties of random labeled trees that hold 
with high probability. These properties are required to effectively use the combinatorial bound of 
Corollary \ref{cor:countpathsIJ}.
As mentioned above, these ``typicality'' properties are not quite sufficient to derive the desired bound for the largest eigenvalue. In Section \ref{sec:treesurgery} we show how random trees can be "safely rewired" using Lemma~\ref{lem:rewiring} so that they become amenable to using our combinatorial tools, without decreasing their top eigenvalue. 
Finally, in Section \ref{sec:endofproof} we put all ingredients of the proof together and describe a code that is sufficiently detailed, allowing us to conclude the proofs of Theorems \ref{thm:main} and~\ref{thm:main2}.

\section{\bf Combinatorial bounds}
\label{sec:combinatorialnew}

\subsection{Decompositions of Dyck paths and the proof of Proposition~\ref{prop:kl_bd}}
\label{sub:decomps}
In this section we elaborate on our method for bounding $N_{\mk}(T)$ via decompositions of Dyck paths. Fix a prefix code $C=(\cw^{(1)},\ldots,\cw^{(a)})$ and a Dyck path $\dyck=(d_0,d_1,\ldots,d_{2\mk})$. 

We begin by defining a partition $P=P(C,\dyck)$ of $[2\mk]$ which encodes ``what codeword each step belongs to''; this is quite similar to how we defined $\vec{t}(C,\dyck)$ in Section~\ref{sub:overview}. 

First, if there is $0 < i < 2k$ such that $d_i=0$ then, recursively, Dyck paths $(d_0,\ldots,d_{i})$ and $(d_{i},\ldots,d_{2k})$ induce partitions $P_1$ and $P_2$ of $\{1,\ldots,i\}$ and $\{i+1,\ldots,2k\}$; in this case we set $P=P_1 \cup P_2$. 

We may now assume that $d_i \ne 0$ for all $0 < i < 2k$. Let $\sigma(0)$ be minimal so that $(d_i,0 \le i \le \sigma(0))$ is an element of $C$, and write $h=d_{\sigma(0)}$.

If $h=0$ then necessarily $\sigma(0)=2k$; in this case we set $P=\{[2k]\}$. Otherwise, for $1 \le i \le h$, let
$\sigma(i) = \min\{m \ge \sigma(i-1): d_m = h-i\}$,
and set 
\[ 
P_0=\big\{\{1,\ldots,\sigma(0)-1,\sigma(0),\sigma(1),\ldots,\sigma(h)\}\big\}. 
\] 
Note that we necessarily have $\sigma(h)=2k$. 
Also, for each $0 \le i < h$, we have $d_{\sigma(i+1)}=d_{\sigma(i)}-1$. It follows that if $\sigma(i+1) > \sigma(i)+1$ then 
\[ 
(d_j-(h-i); ~\sigma(i) \le j \le \sigma(i+1)-1) 
\]
is an excursion (and thus a Dyck path) so, recursively, it induces a partition $P_i$ of $\{\sigma(i)+1,\ldots,\sigma(i+1)-1\}$. 
(For convenience, if $\sigma(i+1)=\sigma(i)+1$ then we take $P_i=\emptyset$.) Now set 
$P(C,\dyck) = \bigcup_{0 \le i < h} P_i$. 

Figure~\ref{fig:dyckpath} contains a graphical representation of the partition of the Dyck path 
\[0121210123232121210
\] by the code $C^*=\{\cw^{(1)},\cw^{(2)},\cw^{(3)}\}=\{{\color{blue}010},{\color{red}0121},{\color{lightgray}0123}\}$. The corresponding partition is 
\begin{align}
     P(C^*,\dyck) &= \{\pi_1,\pi_2,\pi_3,\pi_4,\pi_5\}\nonumber\\
     &=\big\{{\color{red}\{1,2,3,6\}},\color{blue}{\{4,5\}},{\color{lightgray}\{7,8,9,10,13,18\}}, {\color{blue}\{11,12\}}, {\color{blue}\{14,15\}},{\color{blue}\{16,17\}}\big\}\, .\label{eq:partition_example}
\end{align}
For the trivial code $\{01\}$, the partition would instead be
\begin{align*}
     P(\{01\},\dyck) &=\big\{\{1,6\},\{2,3\},\{4,5\},\{7,18\},\{8,13\},\{9,10\},\{11,12\},\{14,15\},\{16,17\}\big\}\, ;
\end{align*}
\begin{figure}[hbt]
    \centering
\includegraphics[width=\textwidth,page=1]{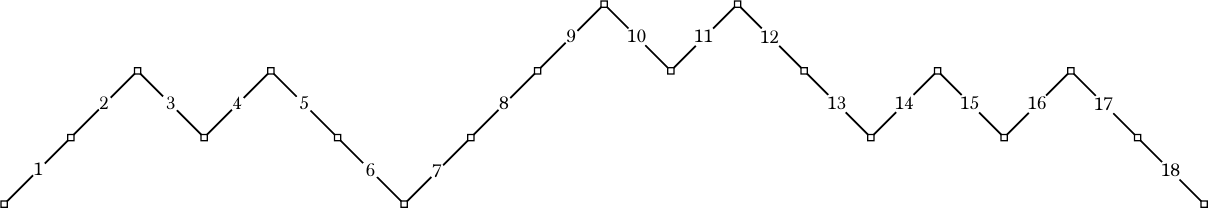}\\
\includegraphics[width=\textwidth,page=2]{graphics/dyckcode.pdf}
    \caption{Top: a Dyck path of length 18. Bottom: a graphical depiction of the partition induced by the code $C^*=\{\cw^{(1)},\cw^{(2)},\cw^{(3)}\}=\{{\color{blue}010},{\color{red}0121},{\color{lightgray}0123}\}$.}
    \label{fig:dyckpath}
\end{figure}

For a closed walk $\rw$ in a tree $T$, we will also write $P(C,\rw):=P(C,D(\rw))$. It is immediate that if two walks $\rw,\rw'$ have $D(\rw)=D(\rw')$, then $P(C,\rw)=P(C,\rw')$. This implies that for any tree $T$, any vertex $v$ of $T$, and any code $C$, we may rewrite $N_{\mk}(v,T)$, the number of walks from $v$ to $v$ of length $2\mk$, as 
\[
\sum_{\{\mathrm{partitions}~P~\mathrm{of}~[2k]\}} 
\sum_{\{\mathrm{Dyck~paths}~\dyck:P(C,\dyck)=P\}} \#\{\mbox{walks }\rw\mbox{ in }T: w_0=w_{2k}=v,D(\rw)=\dyck\}\, .
\]
An easy way to bound this expression from above is to note that whatever $\dyck$ may be, the inner summand is at most $\Delta_T^{\mk}$; this yields  the bound \eqref{eq:trivialbound} from Section~\ref{sub:overview}. To improve this bound, we need to control how much information  required to recover $D(\rw)$ from $P(C,\rw)$ and to recover $\rw$ from $D(\rw)$, and we now turn to this. 

Fix a Dyck path $\dyck$ and a  code $C = \{\cw^{(b)},1 \le b \le a\}$, and let $P=P(C,\dyck)$. Fix any part $\pi$ of $P$, list the elements of $\pi$ in increasing order as $\pi(1),\ldots,\pi(s)$, and let $\pi(0)=\pi(1)-1$. Then we write
$\cw(\pi) = \cw(\pi,C,\dyck)$ for the unique meander $\cw \in C$ which is a prefix of 
\[
(d_{\pi(t)}-d_{\pi(0)},0 \le t \le s)\, .
\]
For example, for the Dyck path in Figure~\ref{fig:dyckpath}, and the partition given in (\ref{eq:partition_example}), we have 
\begin{align}
(c(\pi_1),c(\pi_2),c(\pi_3),c(\pi_4) ,c(\pi_5) ,c(\pi_6))
& =
(0121,010,0123,010,010,010)\nonumber \\
& = (\cw^{(2)},\cw^{(1)},\cw^{(3)},\cw^{(1)},\cw^{(1)},\cw^{(1)}). \label{eq:code_info}
\end{align}
\begin{lem}\label{lem:partition_to_path}
For any code $C$, any Dyck path $\dyck$ is recoverable from the set 
$\{(\pi,\cw(\pi)),\pi \in P(C,\rw)\}$.
\end{lem}
\begin{proof}
Write $\pi_1$ for the part of $P(C,\dyck)$ containing $1$ and $\cw(\pi)=c_1c_2\ldots c_k$. Then listing the elements of $\pi_1$ in increasing order as $i_1,\ldots,i_k$, by construction we have $d_{i_j}=c_j$ for $1 \le j \le k$. The result then follows by induction on the length of $\dyck$, since what remains is to reconstruct $\f(\cw(\pi))+1$ sub-paths of $\dyck$ which are themselves shifted Dyck paths of length strictly less than that of $\dyck$.
\end{proof}

In fact, $\dyck$ can be recovered from somewhat less data than what is used in Lemma~\ref{lem:partition_to_path}, since a part of $\pi$ corresponding to a codeword $\cw$ always begins with $\len(\cw)$ consecutive integers. More precisely, for any part $\pi$ of $P(C,\dyck)$, if $\cw(\pi)=\cw \in C$ then, writing $\pi^{\min} = \min(k:k \in \pi)$, necessarily
\begin{equation}\label{pi_star_trick}
\{\pi^{\min},\pi^{\min}+1,\ldots,\pi^{\min}+\len(\cw)-1\}\subset \pi,
\end{equation}
by the construction of $P(C,\dyck)$. For example, for the Dyck path $\dyck$ in Figure~\ref{fig:dyckpath}, the third part of the partition $P(C^*,\dyck)$ is $\pi_3=\{7,8,9,10,13,18\}$. Given the information that $\cw(\pi_3)=0123$, if we additionally know that $\min(k:k \in \pi_3)=7$ then necessarily also $\{8,9\} \subset \pi_3$, so we may construct $\pi_3$ from the set $\pi_3'=\{7,10,13,18\}$ and the information that $\cw(\pi_3)=0123$. 
Applying this to all of the parts, we see that $\pi$ may be reconstructed from the data in (\ref{eq:code_info}) together with the sets 
\begin{equation}\label{eq:example_piprime}
\{\{1,6\},\{4\},\{7,10,13,18\}, \{11\}, \{14\},\{16\}\}. 
\end{equation}

The example in Figure~\ref{fig:swalk} may be helpful in understanding the coming paragraph.
For a part $\pi$ of $P(C,\dyck)$, writing 
\[
\pi'= \pi\setminus \{k \in \pi: \pi^{\min} < k < \pi^{\min}+\len(\cw(\pi))\}, 
\]
then in the setting of Lemma~\ref{lem:partition_to_path}, the inclusion (\ref{pi_star_trick}) implies that we may recover $\dyck$ from the data 
$\{(\pi',c(\pi)),\pi \in P(C,\dyck)\}$. However, slightly more is true. We may ``compress'' the collection $(\pi',\pi \in P(C,\rw))$ by removing gaps, to form a partition of a {\em consecutive} sequence of integers, without sacrificing recoverability, as follows. Write $\Pi' = \bigcup_{\pi \in P(C,\rw)} \pi'$, and for $\pi \in P(C,\rw)$, let 
\[
\hat{\pi} = \{\mathrm{rank}(i,\Pi'),i \in \pi'\},
\]
where $\mathrm{rank}(i,\Pi') = \#\{j \in \Pi': j \le i\}$. 
For example, if $\Pi'$ is the collection of sets in (\ref{eq:example_piprime}), then $(\hat{\pi},\pi \in P(C,\rw))$ is the partition 
\[
\{\{1,3\},\{2\},\{4,5,7,10\}, \{6\}, \{8\},\{9\}\}\,.
\]
\begin{cor}\label{cor:partition_to_path}
For any code $C$, any Dyck path $\dyck$ is recoverable from the data 
$\{(\hat{\pi},\cw(\pi)),\pi \in P(C,\dyck)\}$.
\end{cor}
\begin{proof}
For each $k \in \bigcup_{\pi \in P(C,\dyck)} \hat{\pi}$, let 
\[
k' = k + \sum_{\{\pi \in P(C,\dyck): \hat{\pi}^{\mathrm{min}}< k\}} (\len(c(\pi))-1).
\]
Then for all $\pi \in P(C,\dyck)$, it holds that $\pi'=\{k':k \in \hat{\pi}\}$, so 
$\{\pi',\pi \in P(C,\dyck)\}$ is recoverable from $\{(\hat{\pi},\cw(\pi)),\pi \in P(C,\dyck)\}$. By the observations preceding the corollary and Lemma~\ref{lem:partition_to_path}, the result follows. 
\end{proof}
This corollary gives us a way of using codes to bound the number of walks in a graph. 
Given a code $C=\{\cw^{(b)},1 \le b \le a\}$, 
and $ b \in [a]$, for a Dyck path $\dyck$ write
\[
P_{b}(C,\dyck) = \{\pi \in P(C,\dyck): \cw(\pi)=\cw^{(b)}\} 
\]
for the set of parts of $P(C,\rw)$ corresponding to the codeword $\cw^{(b)}$.
Then, for $1 \le b \le a$, let 
\[
t_b= t_b(C,\dyck) = |P_{b}(C,\dyck)|\, 
\]
denote the number of times code word $\cw^{(b)}$ is used in the partitioning. We then set $\vec{t}(C,\dyck)=(t_b(C,\dyck),b \in [a])$; this definition agrees with the one from Section~\ref{sub:overview}. 
If $\rw$ is a closed walk in a tree $T$, we also write $P(C,\rw)=P(C,D(\rw))$ and $\vec{t}(C,\rw)=(t_b(C,\rw),b \in [a]):=(t_{b}(C,D(\rw)),b \in [a])$. This allows us to decompose the set of walks $\rw$ of a certain length according to number of occurrences $t_b$ of each code word in $P(C,\rw)$. 

In the next proposition we bound the number of closed walks from a 
vertex to itself with a given ``profile'' $\vec{t}=(t_1,\ldots,t_a)$.  
Recall from Section~\ref{sub:overview} that 
for a vertex $v$ in a tree $T$, we write $N_{C,\vec{t}}(v)$ for the number of closed walks $\rw$ from $v$ to $v$ in $T$ with $t_b(C,\rw)=t_b$ for all $1\le b \le a$. Note that the length of any such walks is $2\mk=2\mk(\vec{t})=\sum_{b=1}^a t_b(\len(\cw^{(b)})+\f(\cw^{(b)}))$. Recall also that for a meander $\cw=(c_0,\ldots,c_j)$, we write 
\[
\Delta_T(\cw) = \max_{v \in V}|\{\mbox{walks $\rw=(w_0,\ldots,w_j)$ in }T: w_0=v, D(\rw)=\cw\}|\,.
\] 
\begin{figure}[ht]
\begin{subfigure}[b]{0.44\textwidth}
\begin{centering}
		\includegraphics[width=\textwidth,page=4]{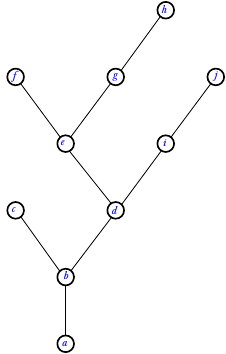}
                \caption{The contour walk $\rw$ of this tree is $abcbdefeghgedijidba$. The corresponding Dyck path $D(\rw)$ is $0121234345432343210$.}
\end{centering}
\end{subfigure}%
\quad
\begin{subfigure}[b]{0.44\textwidth}
\begin{centering}	\includegraphics[width=\textwidth,page=5]{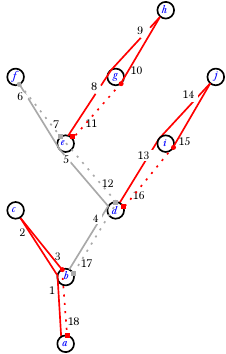}
                \caption{Let $C=\{010,0121,0123\}$; then $P(C,\rw)$ has parts $\{1,2,3,18\}$, $\{4,5,6,7,12,17\}$, $\{8,9,10,11\}$ and $\{13,14,15,16\}$.}
\end{centering}
\end{subfigure}%
\\
\smallskip
\begin{subfigure}[b]{0.48\textwidth}
\begin{centering}
		\includegraphics[width=\textwidth,page=8]{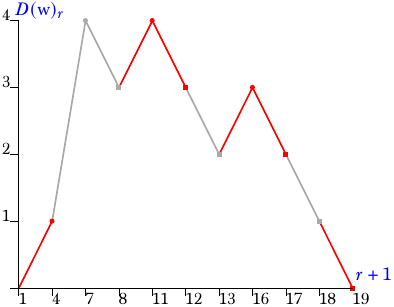}
                \caption{A plot of $(D(\rw)_r:r+1 \in \Pi'\cup \{\len(\rw)+1\})$.}
\end{centering}
\end{subfigure}%
\quad
\begin{subfigure}[b]{0.48\textwidth}
\begin{centering}
		\includegraphics[width=\textwidth,page=6]{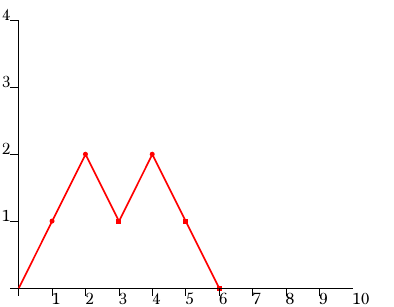}
                \caption{The walk $S$ corresponding to codeword $0121$.}
\end{centering}
\end{subfigure}%
\caption{
For $w$ and $C$ as above, the collection $\{\pi':\pi \in P(C,\rw)\}$ equals $\{\{1,18\},\{4,7,12,17\},\{8,11\},\{13,16\}\}$, and $\{\hat{\pi}:\pi \in P(C,\rw)\}$ equals $\{\{1,10\},\{2,3,6,9\},\{4,5\},\{7,8\}\}$. With $\cw^{(2)}=0121$ and $\cw^{(3)}=0123$, then $\Pi_2 = \{1,4,5,7,8,10\}$ and $\Pi_3 = \{2,3,6,9\}$. 
}\label{fig:swalk}
\end{figure}
\begin{prop}\label{prop:code_bound}
Fix a code $C=\{\cw^{(b)},1 \le b \le a\}$, and
a vector of non-negative integers $\vec{t}=(t_1,\ldots,t_a)$, and let
$q(\vec{t})=\sum_{b=1}^at_b\cdot (1+\f(\cw^{(b)}))$. Then
for any tree $T$ and any vertex $v$ of $T$, 
\begin{equation}\label{eq:code_bound}
N_{C,\vec{t}}(v) \le 
\binom{q(\vec{t})}{t_b{(1+\f(\cw^{(b)}))},1 \le b \le a}\prod_{b=1}^a 
\frac{1}{t_b{(1+\f(\cw^{(b)}))}+1} {t_b{(1+\f(\cw^{(b)}))} +1\choose t_b}\cdot \Delta_T(\cw^{(b)})^{t_b}\, .
\end{equation}
\end{prop}
\begin{proof}
Fix a closed walk $\rw$ in $T$ and a codeword $\cw \in C$. If $\pi \in P(C,\rw)$ and $c(\pi)=\cw$ then $|\pi|=\len(\cw)+\f(\cw)$ and $|\hat{\pi}|=|\pi|-(\len(\cw)-1)=1+\f(\cw)$. It follows that $\hat{P}(C,\rw):=(\hat{\pi},\pi \in P(C,\rw))$ is a partition of $[q(\vec{t})]$. Writing $\Pi_b=\cup_{\pi \in P(C,\rw):\cw(\pi)=\cw^{(b)}}\hat{\pi}$, then $(\Pi_b,1 \le b \le a)$ is another, coarser partition of $[q(\vec{t})]$, in which $|\Pi_b|=t_b(1+\f(\cw^{(b)}))$. Accordingly, the number of possibilities for the partition $(\Pi_b,1 \le b \le a)$, as $\rw$ varies over walks with $(t_b(C,\rw),1 \le b \le a)=\vec{t}$, is at most 
\begin{equation}\label{eq:bigpi_bound}
{q(\vec{t}) \choose t_b{(1+\f(\cw^{(b)}))},1 \le b \le a}. 
\end{equation}

We next focus our attention on the number of words which could give rise to the partition $(\Pi_b,1 \le b \le a)$. Figure~\ref{fig:swalk} should be useful in understanding the arguments of this paragraph.  Consider a given value of $b$, write $\f=\f(\cw^{(b)})$, and define a sequence $S=(S(p),0 \le p \le |\Pi_b|)$ as follows. List the elements of $\Pi_b$ in increasing order as $p_1,\ldots,p_{|\Pi_b|}$. 
Set $S(0)=0$, $S(1)=\f$, and for $1 < i \le |\Pi_b|$ let 
\[
S(i)=	\begin{cases}
			S(i-1)+\f	& \mbox{ if } p_{i}\mbox{ and }p_{i-1}\mbox{ are in different parts of }\hat{P}(C,\rw)\\
			S(i-1)-1		& \mbox{ if } p_{i}\mbox{ and }p_{i-1}\mbox{ are in the same part of }\hat{P}(C,\rw)\, .
		\end{cases}
\]
For any part $\pi$ of $P(C,\rw)$ with $\cw(\pi)=\cw^{(b)}$, we have $|\hat{\pi}|=|\pi'|=h+1$ by construction; from this it is immediate that $S(|\Pi_b|)=0$. 
Moreover, if $\sigma,\rho \in P(C,\rw)$ are distinct and $\rho_j <\sigma^{\min} < \rho_{j+1}$, where $\rho_j$ and $\rho_{j+1}$ are consecutive elements of $\rho$, then in fact $\sigma \subset \{\rho_j+1,\ldots,\rho_{j+1}-1\}$; this is an immediate consequence of the definition of $P(C,\rw)$. This nesting structure implies that there is some part $\pi \in P(C,\rw)$ with $\cw(\pi)=\cw^{(b)}$ such that if $\hat{\pi}^{\min}=p_i$, then 
\[
\{p_i,p_{i+1},\ldots,p_{i+\f}\} = \hat{\pi}. 
\]
We then have
\[
(S(i),S(i+1),\ldots,S(i+\f))=(S(i-1)+\f,S(i-1)+\f-1,\ldots,S(i-1))\, .
\]
It follows by induction that $S$ is non-negative. Moreover, the nesting structure likewise implies that the collection $(\pi \in P(C,\rw),\cw(\pi)=\cw^{(b)})$ can be recovered from $\Pi_b$ and $S$. The number of walks $S$ of length $|\Pi_b|=t_b(1+\f)$ with step sizes in $\{-1,\f\}$, starting and finishing at zero and staying weakly above zero  is precisely $\frac{1}{t(1+\f)+1}{t(1+\f)+1\choose t}$ - see, e.g., \citep[Theorem 10.4.1]{MR3409351}. 

We apply the preceding bound to each set $\Pi_b$. Combined with (\ref{eq:bigpi_bound})  we see that, having constrained $\rw$ to satisfy that $(t_b(C,\rw),1 \le b \le a)=\vec{t}$, the number of possibilities for the collection
\[
\{(\hat{\pi},\cw(\pi)),\pi \in P(C,\rw)\}
\]
is at most 
\begin{equation}\label{eq:dyck_bound}
{q(\vec{t}) \choose t_b{(1+\f(\cw^{(b)}))},1 \le b \le a}\prod_{b=1}^a 
\cdot \frac{1}{t_b{(1+\f(\cw^{(b)}))}+1} {t_b{(1+\f(\cw^{(b)}))}+1 \choose t_b}\, .
\end{equation}
By Corollary~\ref{cor:partition_to_path}, it follows that the number of possibilities for the Dyck path $D(\mathrm{w})$ is bounded by the expression in (\ref{eq:dyck_bound}). 

Finally, to recover $\mathrm{w}$ from $D=D(\mathrm{w})$, it suffices to embed the subwalks corresponding to the parts of $P(C,w)$ into $G$. For a part $\pi$ of $P(C,w)$, having specified the location of $w_{\pi^{\min}-1}$, by definition there are at most $\Delta_T(\cw(\pi))$ possibilities for the subwalk $(w_p,p \in \pi)$. It follows that the number of walks with $D(\mathrm{w})=D$ is at most 
\[
\prod_{b=1}^a \Delta_T(\cw^{(b)})^{t_b}\, ,
\]
which completes the proof. 
\end{proof}
We are now prepared to prove Proposition~\ref{prop:kl_bd}.
\begin{proof}[Proof of Proposition~\ref{prop:kl_bd}]
We bound the multinomial coefficient in (\ref{eq:code_bound}) with the aid of the inequalities $(n/e)^n \le n! \le e\sqrt{n} (n/e)^n$. 
We obtain 
\begin{align*}
{q \choose t_b{(1+\f(\cw^{(b)}))},1 \le b \le a}
& \le e\sqrt{q} e^{-q} q^q \prod_{b=1}^a \frac{1}{(t_b(1+\f(\cw^{(b)})))!}\\
& = 
e\sqrt{q} \prod_{b=1}^a \frac{(q/e)^{t_b(1+\f(\cw^{(b)}))}}{(t_b(1+\f(\cw^{(b)})))!} \\
& \le 
e\sqrt{q}
\prod_{b=1}^a \left(\frac{q}{t_b(1+\f(\cw^{(b)}))}\right)^{t_b(1+\f(\cw^{(b)}))}\, .
\end{align*}
We also have  
\[
\frac{1}{t_b{(1+\f(\cw^{(b)}))}+1}
{t_b{(1+\f(\cw^{(b)}))}+1 \choose t_b}\le
{t_b{(1+\f(\cw^{(b)}))} \choose t_b} \le (1 \vee 2e\f(\cw^{(b)}))^{t_b}\, ,
\]
the second inequality holding 
since both sides are $1$ when $\f(\cw^{(b)})=0$, and when $\f(\cw^{(b)}) \ge 1$ we have 
\[
{t_b{(1+\f(\cw^{(b)}))} \choose t_b} \le
\left(\frac{e t_b(1+\f(\cw^{(b)}))}{t_b}\right)^{t_b} = (e(1+\f(\cw^{(b)})))^{t_b} \le (2e\f(\cw^{(b)}))^{t_b}. 
\]

Using these in the upper bound of Proposition~\ref{prop:code_bound} yields the 
bound
\begin{align*}
N_{C,\vec{t}}(v)
& \le e\sqrt{q} \prod_{b=1}^a \big((1 \vee 2e\f(\cw^{(b)}))\Delta_T(\cw^{(b)})\big)^{t_b} \left(\frac{q}{t_b(1+\f(\cw^{(b)}))}\right)^{t_b(1+\f(\cw^{(b)}))}\\
& 
= 
e\sqrt{q}  \Delta^k
\prod_{b=1}^a 
\left(\frac{(1 \vee 2e\f(\cw^{(b)}))\Delta_T(\cw^{(b)})}
{\Delta^{(\len(\cw^{(b)})+\f(\cw^{(b)}))/2}}\right)^{t_b}
\left(\frac{q}{t_b(1+\f(\cw^{(b)}))}\right)^{t_b(1+\f(\cw^{(b)}))}, 
\end{align*}
where in the second line we have multiplied and divided by 
$\Delta^k=\Delta^{\sum_{b=1}^a t_b(\len(\cw^{(b)})+\f(\cw^{(b)}))/2}$.
Setting $\rho_b = t_b(1+\f(\cw^{(b)}))/q$ for $1 \le b \le a$, we may rewrite the above bound as
\[
e\sqrt{q}  \Delta^k
\exp\Big(
q \sum_{b=1}^a \rho_b \log\big(\frac{g_b}{\rho_b}\big)
\Big)
\]
Since $\sum_{b=1}^a \rho_b = 1$, by the non-negativity of the 
Kullback-Leibler divergence, 
\[
\sum_{b=1}^a \rho_b \log\Big(\frac{g_b}{\rho_b}\Big)
\le 
\log(\sum_{b=1}^a g_b)\, ,
\]
from which the result follows. 
\end{proof}

\subsection{Refined bounds via nonbacktracking walks} 
\label{sec:refined}
As discussed in the proof overview, 
in order to strengthen the above bounds, we need to refine our decomposition of walks in trees according to the times when they visit high-degree vertices. This is achieved by classifying vertices into three types
according to their degrees (``high,'' ``moderate,'' and ``low'' degrees, defined below) and partitioning walks not only by their
distances to the starting vertex but also by the sequence of types of the vertices. The prefix codes may
be easily extended to encode the type information as well. Part of our strategy will be to separately quantify the contribution to the number of closed walks arising from ``backtracking'' subwalks, of type $xyx$ or $xyzyx$, where $x$ is a vertex of medium or high degree and $y$ and $z$ are vertices of low degree. (We shall call these \emph{simple} and \emph{double} backtracking steps, respectively.) 
This will allow us to exploit finer structural properties of uniformly random trees, which we establish in Sections \ref{sec:structural} and \ref{sec:treesurgery}.

Formally, we define an {\em $(M,H)$-tree} to be a $4$-tuple $T=(V,E,M,H)$,
where $(V,E)$ is a tree and $M,H \subset V$ are disjoint sets of vertices; we will eventually take $M$ to be the set of ``moderately high degree vertices'' and $H$ to be the set of ``very high degree vertices.''  Given a walk $\rw=(w_i,0 \le i \le j)$ in an $(M,H)$-tree, we define a {\em type sequence} $\type=(\type_i(T,\rw),0 \le i \le j)$, by setting 
\[
\type_i(T,\rw) =
	\begin{cases}
            h	& \mbox{ if }w_i \in H\\
		m	& \mbox{ if }w_i \in M\\
		\ell	 &\mbox{ if }w_i \in V\setminus (M\cup H) \, .
	\end{cases}
\]

(Here $h,m$, and $\ell$ stand for "high," "moderate", and "low", indicating the degree type of the vertex visited at step $i$ of the walk.) We let $E_T(\rw) = ((\dist_T(w_0,w_i),\type_i(T,\rw)),0 \le i \le j)$, and call $E_T(\rw)$ the {\em encoding path} corresponding to the walk $\rw$. Note that extracting the first coordinate from each of the pairs in the encoding path yields the Dyck path $D_T(\rw)$.

Given a walk $\rw$ in $T$: 
\begin{itemize}
    \item A {\em simple backtracking step} is a subwalk $(w_i,w_{i+1},w_{i+2})$ of $\rw$ such that $w_{i+2}=w_i$; $w_i \in M \cup H$; $w_{i+1} \not \in M \cup H$; and $\dist_T(w_0,w_{i+1})>\dist_T(w_0,w_i)$. 
    \item A {\em double backtracking step} is a subwalk $(w_i,w_{i+1},w_{i+2},w_{i+3},w_{i+4})$ of $\rw$ such that $w_{i+4}=w_i$; $w_i \in M \cup H$; $w_{i+1},w_{i+2} \not \in M \cup H$; and $\dist_T(w_0,w_{i+2})>\dist_T(w_0,w_{i+1})>\dist_T(w_0,w_i)$. 
\end{itemize}
In fact, the condition that $\dist_T(w_0,w_{i+2})>\dist_T(w_0,w_{i+1})$ may be deduced from the other conditions in the definition of double backtracking steps, but we find it clearer to write it explicitly.
An {\em $(M,H)$-doubly-nonbacktracking walk} (or \emph{$(M,H)$-walk} for short) is a walk $\rw$ in $T$ that contains no simple or double backtracking steps.

An {\em $(M,H)$-meander} is a sequence $\dyck=((d_i,\type_i),0 \le i \le j)$ which is the encoding path of some $(M,H)$-walk; it is an excursion if $d_j=0$. We define $\f(\dyck)=d_j$ and call $\f(\dyck)$ the final value of $\dyck$. An $(M,H)$-meander $\dyck$ is an {\em $(M,H)$-excursion} if $\f(\dyck)=0$.
A set $C$ of $(M,H)$-meanders is an $(M,H)$-code if for any $(M,H)$-excursion $\dyck=((d_i,\type_i),0 \le i \le j)$ of positive length, there is a unique $(M,H)$-meander $\cw \in C$ such that $\cw$ is a prefix of $\dyck$.

The ``trivial'' $(M,H)$-code consists of nine words, each of the form $((0,*),(1,*))$, where each entry of ``$*$'' may be replaced by an $\ell,m$, or $h$. Additionally, using $m/h$ to denote an entry which can be either an $m$ or an $h$, the simplest non-trivial $(M,H)$-code has eleven words and is  shown in Figure~\ref{fig:hcode}. Using the above shorthand, it can be written as 
\[
\{((0,\ell),(1,*)),((0,m/h),(1,m/h)),((0,m/h),(1,\ell),(2,m/h))\}\, .
\]

\begin{figure}[ht]
		\includegraphics[width=0.7\textwidth]{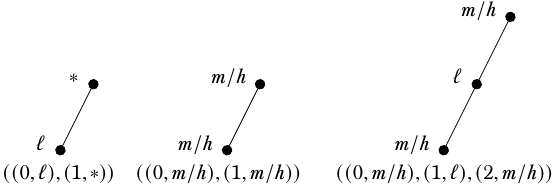}
                \caption{The simplest non-trivial $(M,H)$-code.}
\label{fig:hcode}
\end{figure}

For an $(M,H)$-excursion $\dyck=((d_i,\tau_i),0 \le i \le j)$ and an $(M,H)$-code $C$, the partition $P(C,\dyck)$ is defined just as in Section~\ref{sub:decomps}. 
Fix any part $\pi$ of $P(C,\dyck)$, list the elements of $\pi$ in increasing order as $\pi(1),\ldots,\pi(s)$, and let $\pi(0)=\pi(1)-1$. Then 
there is a unique $(M,H)$-meander $\cw(\pi) = \cw(\pi,C,\dyck) \in C$ such that 
\[
\cw(\pi) 
= ((d_{\pi(r)}-d_{\pi(0)},\type_{\pi(r)}),0 \le r \le s')\, ,
\]
for some $s' \le s=|\pi|$. For a closed walk $\rw$ in an $(M,H)$-tree $T$, we also write $P(C,\rw)=P(C,E(\rw)))$. Like in Section~\ref{sub:decomps}, if walks $\rw,\rw'$ have $E_T(\rw)=E_T(\rw')$, then $P(C,\rw)=P(C,\rw')$. Thus, for any tree $T$, any vertex $v$ of $T$, and any code $C$, we may rewrite $N_{\mk}(v,T)$, the number of walks from $v$ to $v$ of length $2\mk$, as 
\[
\sum_{\{\mathrm{partitions}~P~\mathrm{of}~[2k]\}} 
\sum_{\{(M,H)-\mathrm{excursions}~\dyck:P(C,\dyck)=P\}} \#\{\mbox{walks }\rw\mbox{ in }T: w_0=w_{2k}=v,E(\rw)=\dyck\}\, .
\]

The definitions just preceding Proposition~\ref{prop:code_bound} also carry over naturally to the setting of $(M,H)$-trees. 
For an $(M,H)$-tree $T$, a node $v$ of $T$, and an $(M,H)$-meander $\cw=((c_i,\type_i),0 \le i \le j)$, let 
\[
\mathrm{deg}_T(v,\cw) = \#\{\mbox{$(M,H)$-walks $\rw=(w_0,\ldots,w_j)$ in }T: w_0=v,E_T(\rw)=\cw\}, 
\]
and let $\Delta_T(\cw) = \max\{\mathrm{deg}_T(v,\cw): v\mbox{ a node of }T\}$.

Given an $(M,H)$-code $C=\{\cw^{(b)},1 \le b \le a\}$ 
and a codeword $\cw^{(b)}=((c_i^{(b)},\type_i^{(b)}),0 \le i \le m) \in C$, for an $(M,H)$-walk $\rw$ write 
\[
P_b(C,\rw) = \{\pi \in P(C,\rw): c(\pi)=\cw^{(b)} \} 
\]
for the set of parts of $P(C,\rw)$ corresponding to the $(M,H)$-meander $\cw^{(b)}$. 
Then, for $1 \le b \le a$, let 
\[
t_b= t_b(C,\rw) = |P_{b}(C,\rw)|\, .
\]

Further recycling the notation of Section~\ref{sub:decomps}, given an $(M,H)$-code $C=(\cw^{(b)},1 \le b \le a)$, for a vector of non-negative integers $\vec{t}=(t_1,\ldots,t_a)$ and a node $v$ in an $(M,H)$-tree $T$, 
we denote by $N_{C,\vec{t}}(v)$ the number of $(M,H)$-walks $\rw$ from $v$ to $v$ in $T$ with $t_b(C,\rw)=t_b$ for all $1\le b \le a$. The next proposition bounds the number of such walks from above.

\begin{prop}
\label {prop:ij_code_bound}
Fix an $(M,H)$-tree $T=(V,E,M,H)$ and an $(M,H)$-code $C=\{\cw^{(b)},1 \le b \le a\}$, and write
\begin{equation}\label{eq:epsdef}
    \eps =\eps(C,T) := (2a)^2 \max_{b \in [a]} \frac{\Delta_T(\cw^{(b)})^{\frac{2}{\len(\cw^{(b)})+\f(\cw^{(b)})}}}{\Delta_T}\, 
\end{equation}
Then for any vector $\vec{t}=(t_1,\dots,t_a)\in \Z_{\ge 0}^a$, setting
$\mk(\vec{t}):=\sum_{b \in [a]} t_b(\len(\cw^{(b)})+\f(\cw^{(b)}))/2$, 
we have that
\[
  \max_{v \in T} N_{C,\vec{t}}(v)\leq 2^a\,(\eps\,\Delta_T)^{\mk(\vec{t})}.
\]
\end{prop}

Crucially, the upper bound established for $N_{C,\vec{t}}(v)$ does not depend on $v$, and only depends on $\vec{t}$ via $k(\vec{t})$. 
The form in which we state the bound may seem curious, since we could cancel the factor of $\Delta_T$ with the denominator in \eqref{eq:epsdef}. We wrote the bound in this way because $\eps$ will show up again later, in an expression where such a cancellation is not possible.
Now, for a given value of $\mk\in \N$, there are at most $a^{2\mk}$ choices for the type vector $\vec{t}$ with $\mk(\vec{t})=\mk$. 
Therefore, letting $N^{M,H}_\mk(v,T)$ denote the number of closed $(M,H)$-walks of length $2\mk$ from $v$ to $v$ in $T$ and setting $N^{M,H}_\mk(T)=\max_{v\in V} N^{M,H}_{\mk}(v,T)$,  
we have the following corollary.
\begin{cor}\label{cor:ij_code_bound} 
For any $(M,H)$-tree $T=(V,E,M,H)$ and $(M,H)$-code $C=\{\cw^{(b)},1 \le b \le a\}$, 
for any integer $\mk \ge 0$ we have 
$N^{M,H}_{\mk}(T) \le 2^a(a^2\eps \Delta_T)^{\mk}$, where $\eps=\eps(C,T)$ is as in \eqref{eq:epsdef}. 
\end{cor}
When we later apply this bound, we will need the parameter $\eps$ to be much smaller than $\Delta_T^{-1/2}$; it is there that the right choice of code will be crucial.
\begin{proof}[Proof of Proposition \ref{prop:ij_code_bound}]

Fix $v\in V$ and $\vec{t}$ as above, and define $q=q(\vec{t})=\sum_{b=1}^at_b\cdot (1+\f(\cw^{(b)}))$. Precisely the same proof as for Proposition~\ref{prop:code_bound} gives the bound
\begin{align*}
    & N_{C,\vec{t}}(v) \\
    & \le
{q(\vec{t}) \choose t_b{(1+\f(\cw^{(b)}))},1 \le b \le a}\prod_{b=1}^a 
\frac{1}{t_b{(1+\f(\cw^{(b)}))}+1} {t_b{(1+\f(\cw^{(b)}))} +1\choose t_b}\cdot \Delta_T(\cw^{(b)})^{t_b}\, .
\end{align*}
In order to simplify this bound, note that
\[
  {q(\vec{t}) \choose t_b{(1+\f(\cw^{(b)}))},1 \le b \le a}
  \le a^{q(\vec{t})} \le a^{2\mk(\vec{t})}~,
\]
and that for each $b \in [a]$,
\[
  \frac{1}{t_b{(1+\f(\cw^{(b)}))}+1} {t_b{(1+\f(\cw^{(b)}))} +1\choose t_b}
  \le 2^{t_{b}(1+\f(\cw^{(b)}))+1}~,
\]
leading to
\begin{align*}
  N_{C,\vec{t}}(v)
  & \le a^{2\mk(\vec{t})} 2^{a+2\mk(\vec{t})}  \prod_{b=1}^a \Delta_T(\cw^{(b)})^{t_b}\\
  & = \Delta_T^{\mk(\vec{t})} a^{2\mk(\vec{t})} 2^{a+2\mk(\vec{t})}  \prod_{b=1}^a \left(\frac{\Delta_T(\cw^{(b)})^{\frac{2}{(\len(\cw^{(b)})+\f(\cw^{(b)}))}}}{\Delta_T}     \right)^{t_b (\len(\cw^{(b)})+\f(\cw^{(b)}))/2}~.
\end{align*}
The product is bounded by
\[\left(\frac{\eps}{(2a)^2}\right)^{\sum_{b=1}^at_b(\len(\cw^{(b)})+\f(\cw^{(b)}))} = \left(\frac{\eps}{(2a)^2}\right)^{\mk(\vec{t})}.\]Plugging this back into the previous display finishes the proof.
\end{proof}

\subsection{Bounding the number of walks}

In this section we establish the main combinatorial bound for $N_{\mk}(v,T)$, which we recall is the total number of closed walks of length $2\mk$ from $v$ to $v$ in $T$. 
To this end, first we convert $T$ into an $(M,H)$-tree by
defining $H$ as the set of vertices of degree at least $0.95\Delta_T$ where $\Delta_T$ is the maximum degree in $T$. Moreover, we define $M$ as the set of vertices of degree less than $0.95\Delta_T$ but at least $\kappa$ for a threshold $\kappa>0$ specified later. We say that a vertex is of high degree if $v\in H$, it is of moderate degree if $v\in M$
and it is of low degree otherwise. Recall that $\Delta_T^{(2)}$ is the greatest size of any second neighbourhood in $T$. 

\begin{prop}\label{prop:countpathsIJ}
Fix any real number $\kappa > 0$. Let $T=(V,E,M,H)$ be an $(M,H)$-tree with $H$ the set of vertices of degree at least $0.95\Delta_T$ and $M$ the set of vertices of $V\setminus H$ of degree at least~$\kappa$.
Suppose that there are no adjacent vertices in $T$ of degree at least $0.95\Delta_T$.
Fix an $(M,H)$-code $C=\{\cw^{(b)},1 \le b \le a\}$ and let
\[
    \eps =\eps(C,T)= (2a)^2 \max_{b \in [a]} \frac{\Delta_T(\cw^{(b)})^{\frac{2}{(\len(\cw^{(b)})+\f(\cw^{(b)}))}}}{\Delta_T}
\]
Then for all $v \in V$, 
  \[
    N_{\mk}(v,T) \le
       2^a ({\mk}+1)^2\Delta_T^{{\mk}} \left(1+ \frac{\Delta_T^{(2)}}{\Delta_T^2}\right)^{2{\mk}}
\left(1+ e^{20}a^2\eps\right)^{\mk}~.
  \]
\end{prop}
The proposition has a direct consequence for the largest eigenvalue of $T$. Indeed, since
\[\lambda_1(T) = \lim_{\mk \to \infty}\left(\max_{v\in V}N_{\mk}(v,T)\right)^{\frac{1}{2{\mk}}},\]
 we have the following result. 
\begin{cor}\label{cor:countpathsIJ}In the setting of Proposition \ref{prop:countpathsIJ}, the largest eigenvalue of $T$ satisfies
\[\lambda_1(T)\leq \sqrt{\Delta_T}\,\left(1+ \frac{\Delta_T^{(2)}}{\Delta_T^2}\right)
\,\sqrt{1+ e^{20} a^2\eps}.\]
\end{cor}

To put this into perspective, recall that our goal is to prove an eigenvalue bound of the form $\lambda_1(T)\leq \sqrt{\Delta_T} + o(1)$ for most trees on $n$ vertices. The corollary implies that such a bound follows if both $\Delta^{(2)}_T/\Delta_T$ and $\eps=\eps(C,T)$ are much smaller than $1/\sqrt{\Delta_T}$. While $\Delta^{(2)}_T/\Delta_T$ depends only on the tree, $\eps$ depends also on the choices of code $C$ and threshold $\kappa$ (since $\kappa$ determines $M$ and $H$). It will turn out that taking $\kappa=\log^{1/5}n$ is good enough for our purposes, if the right code is used.

Let us now prove the proposition.

\begin{proof}[Proof of Proposition \ref{prop:countpathsIJ}]

In order to bound $N_{\mk}(v,T)$, we partition the set of walks of length $2{\mk}$ from $v$ to $v$ into equivalence classes. An equivalence class is determined by
\begin{enumerate}
    \item the number of single backtracking steps of type $xyx$ where $x$ has high degree;
    \item the number of single backtracking steps of type $xyx$ where $x$ is of moderate degree;
    \item the number of double backtracking steps of type $xyzyx$ where $x$ is of high degree;
    \item the number of double backtracking steps of type $xyzyx$ where $x$ is of moderate degree;
   
\end{enumerate}

To formalize this, given a walk $\rw=(w_0,\ldots,w_{2\mk})$ in $T$, the {\em reduced $(M,H)$- walk} is obtained from $\rw$ by removing all backtracking steps from $\rw$. In other words, if $(w_i,\ldots,w_{i+j})$ is a simple or double backtracking step of $\rw$, then replace $\rw$ by the subwalk $(w_0,\ldots,w_i,w_{i+j+1},\ldots,w_{2{\mk}})$, and repeat until no simple or double backtracking steps remain. From the definition of backtracking steps, it is straightforward to verify that the resulting walk $\rw^-$ does not depend on the order in which the backtracking steps are removed, and that $\rw ^-$ is an $(M,H)$-walk as defined in the previous subsection.

Writing $\langle\rw^-\rangle$ for the equivalence class of $\rw^-$, which consists of all walks $\rw$ of length $2\mk$ from $v$ to $v$ whose reduced $(M,H)$-walk is $
\rw^-$, we then have
\begin{equation}\label{eq:abstractbd1}
    N_\mk(v,T) =
    \sum_{\rw^-} |\langle \rw^-\rangle|,
\end{equation}
where sum on the right is over all $(M,H)$-walks $\rw^-$ from $v$ to $v$ of (even) length $0 \le 2j \le 2\mk$.

To bound this expression, group $(M,H)$-walks by their lengths. We know from Corollary~\ref{cor:ij_code_bound} that $N_{j}^{M,H}(v,T)$, the number of $(M,H)$-walks of length $2j$ from $v$ to itself, is at most $2^a\,a^{2j}\,\eps^j\,\Delta_T^j$, so 
\begin{equation}\label{eq:abstractbd}
    N_\mk(v,T) \leq 
    \sum_{j=0}^{\mk} (2^a\,a^{2j}\eps^j\,\Delta_T^j)\,Z(j,\mk),
\end{equation}
where 
\[Z(j,\mk):=\max\{|\langle \rw^-\rangle|\,:\,\rw^-\mbox{ is an $(M,H)$-walk of length $2j$}\}.\]

We now make the following claim.

\begin{claim}\label{claim:Z}For each $0\leq j\leq \mk$, \[Z(j,\mk)\leq (\mk+1)^2 \,\,\left(1+\frac{\Delta_T^{(2)}}{\Delta_T^2}\right)^{2\mk}\,\binom{\mk}{j}\,e^{20j}\,\Delta_T^{\mk-j}.\]\end{claim}

The proposition follows from combining this claim with (\ref{eq:abstractbd}) and applying the binomial theorem.
The remainder of the proof thus consists of the proof of Claim \ref{claim:Z}. We fix $0\leq j\leq \mk$ from now on, and perform the following steps.

\begin{itemize}
\item[{\bf Step 1}] We find a function $F:\{\ell,m,h\}^{2j+1}\to \R_+$ with the following property. Consider an $(M,H)$-walk $\rw^-$ of length $2j$ and let $\tau = (\tau_i(C,\rw^-),0 \le i \le {2j})\in \{\ell,m,h\}^{2j+1}$ be the corresponding type sequence. Then:
\[|\langle \rw^-\rangle|\leq F({\tau}).\]
This reduces the problem of bounding $Z(j,m)$ to a maximization of $F$ over a restricted set that we will define below.

\item[{\bf Step 2}] We find an explicit $\tau_{\max}\in \{\ell,m,h\}^{2j+1}$ such that $F(\tau_{\max})\geq F(\tau)$ for any $\tau$ that is the type sequence of some $\rw^-$. We then compute $F(\tau_{\max})$ explicitly to find an upper bound for $Z(j,\mk)$. 
\item[{\bf Step 3}] We finish the proof of the claim by bounding $F(\tau_{\max})$. 
\end{itemize}

Let us now present these steps in detail.

\noindent{\bf Step 1: find a function $F$.} Consider an $(M,H)$-walk $\rw^-$ of length $2j$, and let $\tau = (\tau_i,0 \le i \le j)\in \{\ell,m,h\}^{2j+1}$. Any walk $\rw$ of length $2\mk$ whose reduced $(M,H)$-walk is $\rw^-$ can be constructed in the following way:
\begin{quotation}\noindent
For each $0\leq i\leq 2j$ with $\tau_i\in\{m,h\}$, choose $\sigma_i,\delta_i\in \Z_{\geq 0}$; for $\tau_i=\ell$, set $\sigma_i=\delta_i=0$. For each $0 \le i \le 2j$, we will add $\sigma_i$ single backtracking steps and $\delta_i$ doubly backtracking steps to location $\rw^-_i$ along the walk. To ensure that the total length of the resulting walk is $2\mk$, we require that $\sum_{i=0}^{2j}(\sigma_i+2\delta_i) = \mk-j$.\end{quotation} 

At each location $0\leq i\leq 2j$, having fixed $\sigma_i$ and $\delta_i$, the single and double backtracking steps at location $i$ can happen in a total of up to $\binom{\sigma_i+\delta_i}{\sigma_i}$ possible orders. The number of choices for each single backtracking step is the degree of vertex $\rw_i^-$, which is at most $\Delta_T$ if $\tau_i=h$ and at most $0.95\,\Delta_T$ if $\tau_i=m$. The number of choices for the double backtracking steps is at most the size of the second neighborhood of $\rw_i^{-}$, which we bound by $\Delta^{(2)}_T$. All in all, the number of possible choices for the single and double backtracking steps at position $i$, for a given choice of $\sigma_i$ and $\tau_i$, is at most
\[\binom{\sigma_i+\delta_i}{\sigma_i}\,(\Delta^{(2)}_T)^{\delta_i}\,\Delta_T^{\sigma_i}\,(0.95)^{\sigma_i\I{\type^-_i=m}}.\]
The upshot of the above discussion is this. Given $x = (x_i)_{i=0}^{2j}\in \{\ell,m,h\}^{2j+1}$, define
\[S(x):=\left\{((\sigma_i,\delta_i),0 \le i \le 2j)\in \Z_{\geq 0}^{2(2j+1)}\,:\,
\begin{array}{c}
\sum_{i=0}^{2j}(\sigma_i+2\delta_i) = (\mk-j)\\ \mbox{ and }\\ 
\forall 0\leq i\leq 2j,\, 
\mbox{if $x_i=\ell$ then $\sigma_i=\delta_i=0$}
\end{array}\right
\}
\]
and
\[F(x):=\sum\limits_{(\sigma_i,\delta_i)_{i=0}^{2j}\in S(x)}\,\prod_{i=0}^{2j} \binom{\sigma_i+\delta_i}{\sigma_i}\,(\Delta^{(2)}_T)^{\delta_i}\,\Delta_T^{\sigma_i}\,(0.95)^{\sigma_i\I{x_i=m}}.\]
Then 
\[|\langle \rw^-\rangle|\leq F(\tau).\]

By assumption, $T$ does not contain two adjacent vertices of type $h$. Therefore, $\tau$ cannot contain more than $j+1$ coordinates equal to $h$. This allows us to conclude that
\begin{equation}\label{eq:maxoverrestricted}Z(j,\mk)\leq \max\left\{F(\tau)\,:\,\tau\in\{\ell,m,h\}^{2j+1},\, \sum_{i=0}^{2j}\I{\tau_i=h}\leq j+1\right\}.\end{equation}

\noindent{\bf Step 2: find the maximizer and compute the maximum.}

Inspection reveals that $F$ has the following properties.

\begin{enumerate}
    \item for any $x\in \{\ell,m,h\}^{2j+1}$, $F(x)$ is invariant under permutations of the coordinates of $x$.
    \item Changing some coordinates of $x$ from $\ell$ to $m$ can only increase $F(x)$. Indeed, such a chance to $x$ can only increase the set $S(x)$ and the combinatorial factors in the sum defining $F$.
    \item Likewise, changing some coordinates of $x$ from $m$ to $h$ can only increase $F$.
    \end{enumerate}

Recalling (\ref{eq:maxoverrestricted}), we conclude that

\begin{equation}\label{eq:upperboundF}Z(j,\mk)\leq F(\tau_{\max})\end{equation}
where
\[\tau_{\max} = h^{j+1}m^j.\]
We note in passing that $\tau_{\max}$ is not the type sequence of any $(M,H)$-walk in $T$, as it contains consecutive occurences of $h$. The alternating sequence $hmhmh\dots mh$, which also maximizes $F$, could in principle correspond to a $(M,H)$-walk on some tree. However, working with $\tau_{\max}$ is a bit easier in what follows. 

To continue, we find a more convenient expression for $F(\tau_{\max})$. First of all notice that
\[S(\tau_{\max}):=\left\{(\sigma_i,\delta_i)_{i=0}^{2j}\in \Z_{\geq 0}^{2(2j+1)}\,:\sum_{i=0}^{2j}(\sigma_i+2\delta_i) = (\mk-j)\right\}\]
is the disjoint union of the sets  
\begin{equation}\label{eq:shmd_def}
    S(s_h,s_m,d):=\left\{(\sigma_i,\delta_i)_{i=0}^{2j}\in \Z_{\geq 0}^{2\,(2j+1)}\,:\begin{array}{l}\,\sum_{i=0}^j\sigma_i=s_h,\\ \sum_{i=j+1}^{2j}\sigma_i=s_m, \\ \sum_{i=0}^{2j}\delta_i=d\end{array}\right\}.
\end{equation}
where $(s_h,s_m,d)$ range over all triples in $\Z^3_{\geq 0}$ with $s_h + s_m + 2d =\mk-j.$ Intuitively, $s_h$ and $s_m$ correspond to the number of single backtracking steps at high and moderate vertices, respectively, that were removed to obtain the $(M,H)$-walk, whereas $d$ corresponds to double backtracking steps (the intuition is not quite right, because $\tau_{\max}$ does not correspond to a $(M,H)$-walk, as noted above). 
 
For each element $(\sigma_i,\delta_i)_{i=0}^{2j}\in S(s_h,s_m,d)$, we have
\begin{equation}\label{eq:weightinset}\prod_{i=0}^{2j}(\Delta^{(2)}_T)^{\delta_i}\,\Delta_T^{\sigma_i}\,(0.95)^{\sigma_i\I{t_i=m}} = \Delta_{T}^{s_h}\,(0.95\Delta_T)^{s_m}\,(\Delta_T^{(2)})^{d}.\end{equation} 

On the other hand, the sum 
\[\sum\limits_{(\sigma_i,\delta_i)_{i=0}^{2j}\in S(s_h,s_m,d)}\,\prod_{i=0}^{2j} \binom{\sigma_i+\delta_i}{\sigma_i}\]
counts the number of ways one can write a string of the form
\begin{equation}\label{eq:countstrings}h\,R_0 \,h\,R_2\,\dots\,h\,R_j\,m\,R_{j+1}\,m\,R_{j+2}\,\dots\,m\,R_{2j}\end{equation}
where 
\begin{itemize}
\item each $R_{i}$ is a permutation of $a^{\sigma_i}b^{\delta_i}$, for some choice of $\sigma_i,\delta_i\in\Z_{\geq 0}$ (so there are $\binom{\sigma_i+\delta_i}{\sigma_i}$ choices for $R_i$ once these two numbers is fixed);
\item the total number of characters equal to $a$ in the strings $R_0,\dots,R_j$ is $\sum_{i=0}^j\sigma_i=s_h$;
\item the total number of characters equal to $a$ in $R_{j+1},\dots,R_{2j}$ is $\sum_{i=j+1}^{2j}\sigma_i=s_m$;
\item the total number of characters equal to $b$ in $R_0,R_1,\dots,R_{2j}$ is $\sum_{i=0}^{2j}\delta_i=d$. 
\end{itemize}

We now give an alternative way to count these strings. Notice that such a string has total length $2j+1+s_h+s_m+d$, and always begins with `$h$'. There are thus 
\[\binom{2j+s_h+s_m+d}{d}\]
ways to choose the position of the $d$'s in a string of the above form. 
Once the $d$'s are removed, one is left with a string of the form
\[h\,a^{\sigma_0}\,h\,a^{\sigma_1}\,\dots\,ha^{\sigma_j}\,ma^{\sigma_{j+1}}\,\dots\,ma^{\sigma_{2j}}\]
with $\sum_{i=0}^{j}\sigma_i=s_h$ and $\sum_{i=j+1}^{2j}\sigma_i=s_m$. 
There are therefore
\[\binom{j+s_h}{s_h}\,\binom{j-1+s_m}{s_m}\]
total choices for $(\sigma_0,\ldots,\sigma_{2j})$. The upshot of this discussion is that 
\[\sum\limits_{(\sigma_i,\delta_i)_{i=0}^{2j}\in S(s_h,s_m,d)}\,\prod_{i=0}^{2j} \binom{\sigma_i+\delta_i}{\sigma_i} = \binom{2j+s_h+s_m+d}{d}\,\binom{j+s_h}{s_h}\,\binom{j-1+s_m}{s_m}.\]

Combining the above identity with (\ref{eq:weightinset}) and the definition of $S(s_h,s_m,d)$ in \eqref{eq:shmd_def}, we obtain
\begin{equation}\label{eq:expressionFtmax}F(\tau_{\max}) = \sum\limits_{\stackrel{(s_h,s_m,d)\in \Z^3_{\geq 0}}{s_h + s_m + 2d =\mk-j}}\, \binom{2j+s_h+s_m+d}{d}\,\binom{j+s_h}{s_h}\,\binom{j-1+s_m}{s_m}\,\Delta_{T}^{s_h}\,(0.95\Delta_T)^{s_m}\,(\Delta_T^{(2)})^{d}.\end{equation}

\noindent{\bf Step 3:  bounds and end of proof.} To finish the proof, we find upper bounds for $F(\tau_{\max})$, starting from (\ref{eq:expressionFtmax}), via a series of simplifications and overestimates.

First, notice that the formula can be bounded by noticing that ${\mk}-j-s_{h}-s_m=2d$ and $2j+s_h+s_m+d\leq {\mk}+j-d\leq 2{\mk}$ for all terms in the RHS of (\ref{eq:expressionFtmax}). Pulling out a factor of $\Delta_T^{\mk-j}$, we arrive at
\[
\frac{F(\tau_{\max})}{\Delta_T^{\mk-j}} 
\leq  
\sum
\limits_{\stackrel{(s_h,s_m,d)\in \Z^3_{\geq 0}}{s_h + s_m + 2d={\mk}-j}}\, 
\binom{2\mk}{d}\,
\binom{j+s_h}{s_h}\,
\binom{j-1+s_m}{s_m}\,(0.95)^{s_m}\,
\left(\frac{\Delta_T^{(2)}}{\Delta_T^2}\right)^{d}.
\]
We can further bound the expression, by upper bounding the terms involving $d$ and $s_m$, as 
\[\frac{F(\tau_{\max})}{\Delta_T^{\mk-j}} \leq  \sum\limits_{0\leq s_h\leq \mk-j}\, (\mk+1)\,\binom{j+s_h}{s_h}\,\left[\max_{s_m\in \Z_{\geq 0}}\binom{j-1+s_m}{s_m}\,(0.95)^{s_m}\right]\,\left[\max_{0\leq d\leq 2\mk}\binom{2\mk}{d}\,\left(\frac{\Delta_T^{(2)}}{\Delta_T^2}\right)^{d}\right],\]
where the factor $(\mk+1)$ is an upper bound on the maximum number of choices for $(s_m,d)$ given $s_h$.  

We now show that the maximum in $s_m$ is at most $e^{20j}$. Notice that \[\max_{s_m\in\Z_{\ge 0}}\binom{j-1+s}{s_m}\,(0.95)^{s_m}\]
is achieved at $s_m=s_* := \max\{19\,(j-1)-1,0\}$. For $j=0,1$, the maximum is achieved at $s_*=0$ and equals $1$, which is not larger than $e^{20j}$. For $j\geq 2$, we use the general inequality $\binom{r}{w}\leq (er/w)^w$ (valid for integers $0\leq w\leq r$) and $1+u\leq e^u$ (valid for all real $u$) to obtain:
\begin{align*}\binom{j-1+s_*}{s_*}\,(0.95)^{s_*} & \leq \left[\frac{19e}{20}\,\left(1+\frac{j-1}{19(j-1)-1}\right)\right]^{20\,(j-1)-1} \\ & \leq \exp\left(\frac{19}{20}\,(19\,(j-1)-1) + j-1)\right)\leq e^{20j}.\end{align*}
This gives our desired bound. 

As for the maximum in $d$, we simply note that, by the binomial theorem,
\[\binom{2\mk}{d}\,\left(\frac{\Delta_T^{(2)}}{\Delta_T^2}\right)^{d}\leq \left(1+\frac{\Delta_T^{(2)}}{\Delta_T^2}\right)^{2\mk}.\]

Applying these two bounds, we obtain
\[\frac{F(\tau_{\max})}{\Delta_T^{\mk-j}} \leq  \sum\limits_{0\leq s_h\leq \mk-j}\, (\mk+1) \,\binom{j+s_h}{s_h}\,e^{20j}\,\left(1+\frac{\Delta_T^{(2)}}{\Delta_T^2}\right)^{2\mk}.\]
We are left with a sum with $\mk-j+1\leq (\mk+1)$ terms, and the remaining binomial coefficient in our expression satisfies:
\[\binom{j+s_h}{s_h}=\binom{j+s_h}{j}\leq \binom{\mk}{j}\]
for any $0\leq s_h\leq \mk-j$. We conclude that
\[\frac{F(\tau_{\max})}{\Delta_T^{\mk-j}} \leq  (\mk+1)^2 \,\binom{\mk}{j}\,e^{20j}\,\left(1+\frac{\Delta_T^{(2)}}{\Delta_T^2}\right)^{2\mk}.\]

Combining this with (\ref{eq:upperboundF}), we obtain Claim \ref{claim:Z} and finish the proof.
\end{proof}

\section{\bf Structural properties}
\label{sec:structural}

In this section we establish certain structural properties of random labeled trees that hold with high probability. These properties allow us to make efficient use of the path counting techniques developed in Section \ref{sec:combinatorialnew}.

We begin with a well-known fact, which shall allow us to understand the typical properties of random labeled trees by instead studying Poisson Bienaym\'e trees.\footnote{Bienaym\'e trees are more commonly called {\em Galton--Watson trees}; here we use the alternative terminology proposed in \cite{MR4479914}. Concretely, a Poisson$(\lambda)$ Bienaym\'e tree is the family tree of a branching process with Poisson$(\lambda)$ offspring distribution.} To state the fact, a little more terminology is needed. 
A {\em plane tree}, also called an {\em ordered rooted tree}, is a rooted tree in which the set of children of each vertex are ordered. The vertices of a plane tree are canonically labeled by strings of positive integers as follows: the root is labeled with the empty string $\varnothing$. Inductively, for a vertex labeled $a=a_1a_2\ldots a_k$ and with $m$ children, its children receive labels $a1,a2,\ldots am$ in order.
(Here, for a string $a$ and a positive integer $i$, $ai$ represents the string obtained by concatenating $a$ and $i$.)
In this way, vertices at depth $d$ in the tree are labeled with strings from $\N^d$. The family trees of branching processes are (a particular sort of) random plane trees.

Given a rooted tree $T$ with vertices labelled by $[n]$, the {\em plane tree corresponding to $T$} is constructed as follows. First, for each node $v$ of $T$, order the children of $v$ from left to right in increasing order of vertex label; this endows $T$ with a plane structure. Then remove the original vertex labels.
\begin{lem}\label{lem:pgw0}
Fix $\lambda > 0$ and $n \in \mathbb{N}$. Let $\rt_n^{\bullet}$ be a uniformly random rooted tree with vertices labelled by $[n]$, 
and let $\rt_n^*$ be a Poisson$(\lambda)$ Bienaym\'e tree conditioned to have $n$ vertices. Then the random plane tree corresponding to $\rt_n^{\bullet}$ has the same distribution as $\rt_n^*$.
\end{lem}
For a proof of the lemma, see \cite[Section 5]{MR0458613} or \cite[Section 7]{MR1630413}. 

In the remainder of the section, we bound the probability that $\rt_n$ has certain structural properties by first proving stronger bounds on the probability that a Poisson$(1)$ Bienaym\'e tree has such properties, then using Lemma~\ref{lem:pgw0} (among other tools) to transfer those bounds to $\rt_n^\bullet$ (and thence to $\rt_n$, when needed, by a union bound over the choices of possible root). 

A second useful tool is the {\em breadth-first search} (BFS) construction of plane trees. Given a sequence of non-negative integers $\mathrm{d}=(d_i,i \ge 0)$, the BFS construction works as follows. First, $d_0$ is the number of children of the root. Next, $d_1,\ldots,d_{d_0}$  are the numbers of children of the children of the root, ordered from left to right. The construction continues sequentially in this manner; so, for example, the first child of the root has children with $d_{d_0+1},\ldots,d_{d_0+d_1}$ children, the second has children which themselves have $d_{d_0+d_1+1},\ldots,d_{d_0+d_1+d_2}$ children, and so on. The construction may continue indefinitely or may halt; if it halts, it does so at step $\sigma=\sigma(\mathrm{d})=1+\inf\{m: d_0+\ldots+d_m \le m\}$, and in this case the resulting tree has $\sigma$ nodes. 

Given a rooted tree $T$ and $k \in \N$, write $T(k)$ for the set of nodes at distance $k$ from the root of $t$. Note that the BFS construction imparts a total order to the nodes of the resulting tree, which agrees with the order of the degrees in the sequence $\mathrm{d}$: if $k < l$ then nodes of $T(k)$ precede those of $T(l)$, and within $T(k)$ nodes are ordered from left to right. 

The following construction of Bienaym\'e trees is classical. 
\begin{lem}\label{lem:gw_bfs}
Let $\mathrm{X}=(X_i,i \ge 0)$ be IID non-negative integer-valued random variables with law $\mu$. Then the tree $\rt$  built by applying the BFS construction to $\mathrm{X}$ is a Bienaym\'e$(\mu)$ tree.
\end{lem}
Let $S_0=0$ and, for $k \ge 1$ let $S_k=
\sum_{0 \le i < k} (X_i-1)$. 
Then $\sigma = \inf\{k: S_k =-1\}$ is the size (number of nodes) of $\rt$, and $(1+S_k,0 \le k \le \sigma)$ is the {\em BFS queue length process}, tracking the number of nodes which have been revealed in the breadth-first search process but whose degree in $\rt$ is not yet known.

In the remainder of the section we let $\mathrm{X}=(X_i,i \ge 0)$ be a sequence of independent Poisson$(1)$ random variables and let $\rt$ be the Poisson$(1)$ Bienaym\'e tree built by applying the BFS construction to $\mathrm{X}$.
\begin{lem}\label{lem:gen_width}
For all $k,s \in \N$, 
\[
\p{\exists i \le k: |\rt(i)| \ge s} \le \p{\max_{i \le (k-1)s} S_i \ge s}\, .
\]
\end{lem}
\begin{proof}
Let $\tau(0)=0$ and, for $i \ge 0$, let $\tau(i+1)=\tau(i)+S_{\tau(i)}+1$. It is straightforward to see, by induction, that for all $i \ge 0$, $\tau(i) = \sum_{0 \le j < i} |\rt(j)|$ and $|\rt(i)|=S_{\tau(i)}+1$.

Now let $I = \inf\{j\ge 0: |\rt(j)| \ge s\}$. Since $|\rt(0)|=1$ and $|\rt(j)| < s$ for $0<j < I$, it follows that $\tau(I) \le (I-1)(s-1)+1 \le (I-1)s$ and $S_{\tau(I)} \ge s$. Thus 
if $I \le k$ then there is $j \le (k-1)s$ such that $S_j \ge s$. The result follows. 
\end{proof}
\begin{prop} \label{prop:stopping}
Fix positive functions $m=m(n) =o(n)$ and 
$r=r(n) =o(n^{1/2})$. 
If $\tau$ is any stopping time for $\mathrm{S}=(S_k,k \ge 0)$, then 
\[
\probC{\tau \le m}{|\rt|=n} = (1+o(1)) \E{S_\tau \I{\tau \le \min(m,|\rt|)} \I{S_\tau \le r}} + 
\probC{\tau \le m,S_\tau > r}{|\rt|=n}
\]
\end{prop} 
\begin{proof} 
We first write 
\begin{equation}\label{eq:prop_stopping}
\p{\tau \le m, |\rt|=n} 
= \p{\tau \le m, S_\tau \le r,|\rt|=n} + \p{\tau \le m, S_\tau > r, |\rt|=n}\, .
\end{equation}
To bound the first probability on the right, note that since $|\rt|=\sigma$, 
\begin{align*}
\p{\tau \le m, S_\tau \le r,|\rt|=n} & = \p{\tau \le m, S_\tau \le r,\sigma=n}\\
	& = \sum_{i \le m} \sum_{s = 1}^r\p{\tau =i \le \sigma, S_\tau=s,\sigma=n} \\
	& = \sum_{i \le m} \sum_{s =1}^r\p{\tau =i\le \sigma, S_i=s}\probC{\sigma=n}{\tau =i\le \sigma, S_i=s} 
\end{align*}
By Kemperman's formula \cite[6.3]{MR2245368}, 
\[
\probC{\sigma=n}{\tau =i\le \sigma, S_i=s} = \frac{s}{n} \p{S_{n-i}=-s}\,
\]
so
\begin{align*}
\p{\tau \le m, S_\tau \le r,|\rt|=n}
& =
\sum_{i \le m} \sum_{s = 1}^r\p{\tau=i\le \sigma, S_i=s} \cdot \frac{s}{n} \p{S_{n-i}=-s} \, .
\end{align*}
Uniformly over $i$ and $s$ with $i \le m=o(n)$ and $s \le r=o(n^{1/2})$, by the local central limit theorem $\p{S_{n-i}=-s} = (1+o(1)) \p{S_n=-1}$. Combined with Kemperman's formula this gives that 
\[
\frac{s}{n}\p{S_{n-i}=-s} = (1+o(1)) \frac{s}{n}\p{S_n=-1} =  (1+o(1)) s\p{|\rt|=n}\, .
\]
It follows that 
\begin{align*}
\p{\tau \le m, S_\tau \le r,|\rt|=n}
& = 
(1+o(1))\sum_{i \le m} \sum_{s = 1}^r \p{\tau =i\le \sigma,S_i=s} \cdot s \p{|\rt|=n} \\
& = 
(1+o(1)) \E{S_\tau \I{\tau \le \min(m,\sigma)} \I{S_\tau \le r}} \cdot \p{|\rt|=n}\, .
\end{align*}
Using this identity, the proposition now follows from (\ref{eq:prop_stopping}) and Bayes formula. 
\end{proof}

The next lemma controls neighbourhood sizes near a fixed vertex of $\rt$; the following lemma uses the first to deduce additional information about the local structure of $\rt$. 
\begin{lem}\label{lem:pbbounds}
Let $\rt$ be a Poisson$(1)$ Bienaym\'e tree. 
Then for $n$ sufficiently large, the following bounds both hold.
First, 
\[
\p{\exists j \le 4: |\rt(j)|\ge 37\log n} \le n^{-4}.
\]
Second,
\begin{equation}\label{eq:big_nbhd_bd0}
\p{\exists j \le \log n: |\rt(j)| \ge \lfloor \log^2 n\rfloor}
\le
2 \exp(-(\log^2 n)/5)\, .
\end{equation}
\end{lem}
\begin{proof}
By Lemma~\ref{lem:gen_width}, we have 
\[ 
\p{\exists j \le k: |\rt(j)| \ge s} 
\le 
\p{\max_{i \le (k-1)s} S_i \ge s} 
\] 
where $(S_k,k \ge 0)$ is a random walk with steps $(X_i-1,i \ge 1)$. 
Since a Poisson$(\lambda)$ random variable has a median in $[\lambda-\ln 2,\lambda+1/3)$ (see \cite{MR1195477}) and $X_1+\cdots+X_i$ is Poisson$(i)$-distributed, 
by L\'evy's reflection principle \cite[Theorem 1.4.13]{MR2760872}, for $s \ge 1$ we have 
\[
\p{\max_{i \le (k-1)s} S_i \ge s}
 \le 2\p{S_{(k-1)s} \ge s-1}.
\]
We now use a standard Poisson tail bound \cite[Lemma 1.2]{MR1986198}: for $\eps > 0$, 
\begin{equation}\label{eq:ptb}
\p{\mathrm{Poi}(\lambda) >(1+\eps)\lambda)}
\le \exp(-\lambda((1+\eps)\log(1+\eps)-\eps))\, .
\end{equation}
If $\eps \le 1/2$ then this upper bound is at most $\exp(-\eps^2\lambda/4)$. 
If $k \ge 3$, so that $s-1 \le (k-1)s/2$, then using that $S_i+i$ is Poisson$(i)$-distributed, this implies that 
\[
\p{S_{(k-1)s} \ge s-1}
\le 
2\p{\mathrm{Poisson}((k-1)s) \ge 3(k-1)s/2} \le 
\exp\left(-\frac{1}{4}\frac{(s-1)^2}{(k-1)s}\right)\, .
\]
Taking $k=4$ and $s=37\log n$ now yields the first bound, and taking $k=\lfloor \log n\rfloor$ and $s=\lfloor \log^2 n\rfloor$ yields the second bound.
\end{proof}
\begin{lem}\label{lem:tprops}
Let $\rt$ be a Poisson$(1)$ Bienaym\'e tree. Then for $n$ sufficiently large, the following bounds both hold.
First, writing $\deg(w)=\deg_{\bf T}(w)$ for vertices $w$ of ${\bf T}$, and setting 
\[
K^*:=\min(k \in \N: \mathrm{deg}(w) < \log^{1/5}n \mbox{ for all } w \in \rt(2k) \cup \rt(2k+1)),
\]
we have
\[
\p{K^* \ge 15(\log^{4/5}n)/\log\log n} < 2n^{-4}\, .
\]
Second, writing
\[
Z = \{v \in \rt: |v|< 2K^*, \mathrm{deg}(v) \ge \log^{1/5} n\},
\]
we have
\[
\p{\sum_{v \in Z} \mathrm{deg}(v) \ge 3 \log n} < 4n^{-4}\, ,
\]
and 
\[
\p{\sum_{v \in Z} |\{w \in [n]: \mathrm{dist}(v,w) \le 3\}| \ge 150(\log n)(\log\log n)^2} < 7n^{-4}.
\]
\end{lem}
\begin{proof}
Let $\tau_0=0$ and for $i \ge 0$ let $\tau_{i+1}=\tau_i+S_{\tau_i}+1$. Note that $\tau_i-1$ is a stopping time for the filtration generated by $\mathrm{X}$. Writing $\f(\rt)$ for the height of $\rt$, then for $0 \le k < \f(\rt)$, the $k$'th generation $\rt(k)$ of $\rt$ is constructed by BFS between times $\tau_k$ and $\tau_{k+1}$, and $\f(\rt)=\inf(k> 0: \tau_{k}=\tau_{k-1})-1$ is the height of $\rt$. In particular, for $0 \le k < \f(\rt)$, the degrees of the vertices of $\rt(k)$ are $X_{\tau_k},\ldots,X_{\tau_{k+1}-1}$. 

For $k \ge 0$ let 
\[
E_k = \{\max(X_{i},\tau_{2k}\le i \le \min(\tau_{2k+2}-1,\tau_{2k}+2\lfloor\log^2 n\rfloor) < \log^{1/5}n-1\}.
\]
Since $\tau_{2k}-1$ is a stopping time, by the Markov property we have 
\begin{align*}
& \p{E_k^c~|~X_i,0 \le i \le \tau_{2k}-1}\\
& =\p{\max(X_j,1 \le j \le \min(\tau_2-1,1+2\lfloor \log^2 n\rfloor)
\ge \log^{1/5} n-1}\\
& \leq \p{\max(X_j,1 \le j \le 1+2\lfloor \log^2 n\rfloor\rfloor)
\ge \log^{1/5} n-1} \\ 
& \le \frac{4\log^2 n}{(\lceil\log^{1/5} n -1\rceil)!}\\
& \le 
4\exp
\left(
(11/5)\log\log n-(\log^{1/5}n - 1)\log(\log^{1/5}n-1)
\right)\end{align*}
where in the third inequality we have used the bound $m! > (m/e)^m$. 

Now let $K = \min(k \ge 0: E_k\mbox{ occurs})$. 
Setting $k^* = C(\log^{4/5}n)/\log \log n$ 
for $C$ sufficiently large ($C=15$ suffices), it follows from the above bound and the Markov property that that for all large $n$,
\begin{align*}
 \p{K \ge k^*}
& \le \left(4\exp
\left(
(11/5)\log\log n-(\log^{1/5}n - 1)\log(\log^{1/5}n-1)
\right)\right)^{k^*}\\
& < n^{-4}\, .
\end{align*}

Since the degree of a vertex in $\rt$ is at most its number of children plus one, 
on the event that $|\rt(j)| \le \log^2 n$ for all $j \le \lfloor \log n\rfloor$, 
if $K < k^*$ the maximum degree in $\rt(2K)\cup \rt(2K+1)$ is less than $\log^{1/5}n$, so also $K^* < k^*$.
It thus follows by the above tail bound on $K$ together with \eqref{eq:big_nbhd_bd0} that 
\[
\p{K^* \ge k^*} \le 2\exp(-(\log^2 n)/5) + n^{-4},
\]
which establishes the first bound.

Next, for $i \ge 1$ write $X_i^+ = (X_i+1)\I{X_i \ge \log^{1/5} n-1}$. 
On the events that 
$|\rt(j)|< \log^2 n$ for all $j \le \lfloor \log n\rfloor$ and that $K^* \le 15\log^{4/5} n/\log\log n$, which have combined probability at least $1-2n^{-4}$ for $n$ large, we have 
\[
\sum_{v \in Z} \mathrm{deg}(v)
\le
\sum_{i < \lfloor \log^2 n \rfloor \cdot 30 \log^{4/5} n/\log\log n}(X_i^++1) \le 
\sum_{i \le \log^{14/5}n} (X_i^++1).
\]
To control the latter sum, 
write $S^+:=\sum_{i \le \log^{14/5}n} X_i^+$. 
For $t > 0$, a Chernoff bound gives 
\begin{align*}
\p{S^+ \ge M} 
& \le e^{-tM}\E{e^{tS^+}}\\
& = e^{-tM}(\e{e^{tX_1^+}})^{\log^{14/5}n}\, .
\end{align*}
For any $j \ge 2$ it holds that 
$\E{\exp(t(X_1+1)\I{X_1 \ge j})} \le 1+2e^{tj}/j!$, so taking $t=2$ and $j=\lceil \log^{1/5}n-1\rceil$ 
the above bound gives 
\[
\p{S^+ \ge M} 
\le 
e^{-2M} \pran{1+\frac{2e^{2\lceil \log^{1/5}n-1\rceil}}{(\lceil \log^{1/5}n-1\rceil)!}}^{\log^{14/5} n}
\]
Using that $m!\ge (m/e)^m$ and that $1+x \le e^x$ for $x > 0$, 
it follows that 
\[
\pran{1+\frac{2e^{\lceil \log^{1/5}n-1\rceil}}{(\lceil \log^{1/5}n-1\rceil)!}}^{\log^{14/5} n} < 2
\]
for $n$ large, so for $n$ large we have 
\[
\p{S^+ \ge M} < 2e^{-2M}.
\]
Taking $M = 2\log n$, the bound on $\p{\sum_{v \in Z} \mathrm{deg}(v) \ge 3 \log n}$ follows. 

For the third bound, 
we claim that for $n$ large, for any $K \in \N$ and $C \ge 7\log\log n$ with $KC\log C \ge 8\log n$, 
\begin{equation}\label{eq:newbd}
\p{\exists S\subset [\lfloor\log^3 n\rfloor]: |S| \le K, \sum_{i \in S} X_i \ge CK} \le n^{-4}\, .
\end{equation}
Assuming this, we prove the third bound as follows. Let $S$ be the set of times at which neighbours of vertices in $Z$ (including parents of vertices of $Z$) are explored in the breadth-first search process. On the event $G_1$ that  
$|\rt(j)| < \log^2 n$ for all $j \le \lfloor \log n\rfloor$ and that $K^* < 15\log^{2/3} n/\log\log n$, we have $S \subset [\lfloor\log^3 n\rfloor]$. On the event $G_2$ that $\sum_{v \in Z} \mathrm{deg}(v) < 3\log n$, this set $S$ also has size less than $3\log n$. It thus follows from the preceding bounds of the lemma and from \eqref{eq:newbd} applied with $K=\lfloor 3\log n\rfloor$ and $C=7\log\log n$ that
\[
\p{G_1,G_2,\sum_{i \in S} X_i \ge 3\log n\cdot 7\log\log n} \le n^{-4}\, .
\]
Let $G_3$ be the event that $\sum_{i \in S} X_i <3\log n\cdot 7\log\log n$. On $G_1\cap G_2 \cap G_3$, the second neighbourhood of $Z$ has size at most $3\log n\cdot (7\log\log n+1)$. 
Moreover, writing $S_2$ for the set of times at which vertices in the second neighbourhood of $Z$ are explored, on $G_1$ we have $S_2\subset[\lfloor \log^3 n\rfloor]$. It follows from \eqref{eq:newbd} applied with $K=\lfloor 3\log n(7\log\log n+1)\rfloor$ and $C=7\log\log n$ that 
\[
\p{G_1,G_2,G_3,\sum_{i \in S_2} X_i \ge 3\log n\cdot (7\log\log n+1)7\log\log n} \le n^{-4}. 
\]
For $n$ large, if $\sum_{i \in S_2} X_i < 3\log n\cdot (7\log\log n+1)7\log\log n$ and $G_2$ and $G_3$ both occur then 
\[
\sum_{v \in Z} |\{w \in [n]: \mathrm{dist}(v,w) \le 3\}| < 150 (\log n)(\log\log n)^2,
\]
so the above bounds yield that 
\begin{align*}
\p{\sum_{v \in Z} |\{w \in [n]: \mathrm{dist}(v,w) \le 3\}| < 150 (\log n)(\log\log n)^2}
& < 2\exp(-(\log^2 n)/5)+6n^{-4} \\
& < 7 n^{-4},
\end{align*}
the last bound holding for $n$ large, as required.

It remains to prove \eqref{eq:newbd}; for this write $L=\lfloor \log^3 n \rfloor$. 
It  suffices to consider sets $S\subset [L]$ of size exactly $K$, since the $X_i$ are non-negative.
By a union bound, it then follows that 
\begin{eqnarray*}
\lefteqn{
\p{\exists S\subset [\lfloor\log^3 n\rfloor]: |S| \le K, \sum_{i \in S} X_i \ge CK} } \\
& \le & {L \choose K} \p{\mathrm{Poisson}(K) \ge CK} \\
& \le & 
\exp\left(K\log L-K(C\log C-(C-1))\right),
\end{eqnarray*}
where in the second inequality we have used that ${L \choose K} \le L^K$ and the Poisson upper tail bound~\eqref{eq:ptb}. 

Since $C > 7\log\log n$ and $\log L \le 3\log\log n$, for $n$ large we have $C\log C - (C-1)-\log L > C\log C/2$ and so the final bound is at most $e^{-(KC\log C)/2} < n^{-4}$. 
\end{proof}

The next corollary transfers the above probability bounds from $\rt$ to $\rt_n$. 
For a vertex $u$ of $\rt_n$ and a positive integer $j$, define $\Gamma_{j}(u,\rt_n)$ (or $\Gamma_j(u)$ for simplicity)  as the set of vertices $v$ with $\mathrm{dist}_{\rt_n}(u,v)=j$ where we recall from the introduction that $\mathrm{dist}_{\rt_n}(u,v)$ is the distance of $u$ and $v$ in the tree. Similarly, $\Gamma_{\leq j}(u,\rt_n)$ (or simply $\Gamma_{\leq j}(u)$) denotes the set of all vertices $v$ with $\dist_{\rt_n}(u,v)\leq j$. 
\begin{cor}\label{cor:typicalclusterdegrees}
For $n \ge 1$ let $\rt_n$ be a uniformly random tree with vertices labeled by $[n]$. 
Then there exists a constant $C>0$ such that the following bounds hold for all $n$.
\begin{equation}\label{eq:big_nbhd_bd}
\p{\exists u \in [n],\exists j \le \log n: |\Gamma_{j}(u)| \ge \lfloor \log^2 n\rfloor}
\le
Cn^{5/2} \exp(-(\log^2 n)/5)\, .
\end{equation}
Next, for $u \in [n]$, writing
\[
K^*(u):=\min(k \in \N: \mathrm{deg}(v) < \log^{1/5}n \mbox{ for all } v \in N_u(2k) \cup N_u(2k+1)),
\]
then 
\[
\p{\max(K^*(u),u \in [n]) \ge 15(\log^{4/5}n)/\log\log n} \le Cn^{-3/2}\, .
\]
Finally, for $u \in [n]$ writing 
\[
Z(u) = \{v \in \rt_n: \mathrm{dist}(u,v)< 2K^*(u), \mathrm{deg}(v) \ge \log^{1/5} n\},
\]
then
\[
\p{\max_{u \in [n]}\sum_{v \in Z(u)} \mathrm{deg}(v) \ge 3 \log n} \le Cn^{-3/2}\, ,
\]
and
\[
\p{\max_{u \in [n]}\sum_{v \in Z(u)} |\Gamma_{\le 3}(v)| \ge 150(\log n)(\log\log n)^2} < Cn^{-3/2}
\]
\end{cor}
\begin{proof}
We continue to take $\rt$ as above. Recalling that $\sigma=\inf(k:S_k =-1)$, then 
\begin{equation}\label{eq:treesizeprob}
\p{|\rt|=n} = \p{\sigma=n} = \frac{1}{n} \p{X_1+\ldots+X_n=n-1} = \frac{1+o(1)}{\sqrt{2\pi}n^{3/2}}\, ,
\end{equation}
the first identity holding by another application of Kemperman's formula and the second by the local central limit theorem (or by Stirling's formula). Combining this fact with Lemma~\ref{lem:pgw0}, it follows that for any graph property $\mathcal{P}$, 
\begin{align*}
\p{\rt_n^{\bullet}\mbox{ has property }\cP}
& = \Cprob{\rt\mbox{ has property }\cP}{|\rt|=n}\\
& \le (1+o(1))\sqrt{2\pi} n^{3/2}\cdot 
 \p{\rt\mbox{ has property }\cP}\, .
\end{align*}
Recall that $\rt_n^{\bullet}$ is a uniformly random rooted tree with vertices labeled by $[n]$. We may generate $\rt_n^{\bullet}$ from $\rt_n$ by choosing a root $\rho$ uniformly at random, so it follows by a union bound and the above asymptotic inequality that
\begin{align*}
& \p{\exists u \in [n],\exists j \le \log n: |\Gamma_j(u)| \ge \lfloor \log^2 n\rfloor}\\
& \le
n
\p{\exists j \le \log n: |\rt_n^{\bullet}(j)| \ge \lfloor \log^2 n\rfloor} \\
&
\le (1+o(1))\sqrt{2\pi} n^{5/2}
\p{\exists j \le \log n: |\rt(j)| \ge \lfloor \log^2 n\rfloor}.
\end{align*}
The first bound of the corollary now follows from the second bound of Lemma~\ref{lem:pbbounds}. 
The other bounds of the corollary follow from  Lemma~\ref{lem:tprops} in an essentially identical fashion. 
\end{proof}
We also record the following tail bound on the largest degree in $\rt_n$, which will be used, when proving the main theorem, to convert probability tail bounds to expectation bounds.
\begin{fact}\label{fact:maxdeg}
For all $n$ sufficiently large, $\p{\Delta_{\rt_n} \ge 2\log n} \le n^{-2}$.
\end{fact}
\begin{proof}
An argument just as in the above corollary shows that 
\begin{align*}
    \p{\Delta_{\rt_n} \ge 2\log n}
& \le (1+o(1))\sqrt{2\pi}n^{5/2}
\p{|\rt(1)|\ge 2\log n}\\
& = (1+o(1))\sqrt{2\pi}n^{5/2} \p{\mathrm{Poisson}(1) \ge 2\log n} \le n^{-2}\, ,
\end{align*}
the last bound holding for $n$ large.
\end{proof}
We also need to control the number of vertices of high degree contained in a small neighbourhood. For this, we use Proposition~\ref{prop:stopping}.

\begin{lem}\label{lem:adjacent} The probability that $\rt_n$ contains two adjacent nodes of degree $\geq 0.9\log n/\log \log n$ is at most $n^{-0.8+o(1)}$.
\end{lem}
\begin{proof}Set $h:=0.9\log n/\log\log n$. In the proof we will use the following estimates,
\begin{equation}\label{eq:tailbounds}\forall i\geq 0\,:\,\p{X_i\geq h-1} \leq  n^{-0.9+o(1)} \mbox{ and }\p{X_i\geq \log n}=n^{-\omega(1)},\end{equation}
which both follow from the Poisson tail bound (\ref{eq:ptb}).

Now let $E$ denote the event where there are two adjacent nodes in $\rt_n$ with degrees $\geq h$. Recall that the random rooted $n$-vertex tree $\rt^{\bullet}_n$ can be obtained from $\rt_n$ by picking a root uniformly at random, so 
\[\frac{\p{E}}{n}\leq \p{\deg_{\rt^{\bullet}_n}({\rm root})\geq h\mbox{ and }\max_{v\in {\rt^{\bullet}_n}(1)} \deg_{\rt^{\bullet}_n}(v)\geq h}=\p{F\mid |\rt|=n},\]
where
\[F:=\left\{\deg_{\rt}({\rm root})\geq h\mbox{ and }\max_{v\in {\rt}(1)} \deg_{\rt}(v)\geq h\right\} = \left\{X_0\geq h\mbox{ and }\exists\, 1\leq i\leq X_0\,:\,X_i\geq h-1\right\}.\]
Therefore, our goal is to show that $\p{F\mid |\rt|=n} \leq  n^{-1.8+o(1)}$. To this end, consider the events
\[G_{0}:=\{\exists 0\leq i\leq \log n \,:\, X_i\geq \log n\}\]
and
\[G_{1}:=\{X_0\in [h,\log n]\mbox{ and }\exists 1\leq i\leq \log n \,:\, X_i\in [h-1,\log n]\}.\]
Then $F\subset G_0\cup G_1$. Estimate (\ref{eq:tailbounds}) and a union bound imply that $G_0$ has superpolynomially small probability in $n$ if we do not condition on $\p{|\rt|=n}$. Since $\p{|\rt|=n}\asymp n^{-3/2}$ by (\ref{eq:treesizeprob}), this also holds conditionally. It thus suffices to show that $\p{G_1\mid |\rt|=n}\leq n^{-1.8+o(1)}$, which we do via Proposition \ref{prop:stopping}.

Define a stopping time $\tau$ for the $(X_i)_{i\geq 0}$ (or equivalently, for $(S_i)_{i\geq 0})$) as follows. If $X_0\not \in [h,\log n]$, or if $X_i\not\in [h,\log n]$ for all $1\leq i\leq n$, we set $\tau=n$. Otherwise, we let $\tau$ be the smallest index $1\leq i\leq n$ with $X_i\in [h,\log n]$. Notice that $G_1 = \{\tau\leq \log n\}$. Applying Proposition \ref{prop:stopping} with $d= (\log n+1)^2$ and $m=\log n$, we obtain
\begin{eqnarray*}\p{G_1\mid |\rt|=n} &=& (1+o(1))\E{S_{\tau}\I{S_\tau\leq d,\tau\leq \min(\log n,|\rt|)}} + \p{\tau\leq \log n,S_\tau>d\mid |\rt|=n}\\ &= & O(\log^2n) \p{\tau\leq \log n} + \p{\tau\leq \log n,S_\tau>(\log n+1)^2\mid |\rt|=n}.\end{eqnarray*}
The first term on the right is 
\[O(\log^2n)\,\p{\tau\leq \log n}= O(\log^2n)\,\p{X_0\geq h\mbox{ and }\exists 1\leq i\leq \log n,\, X_i\geq h-1}\leq n^{-1.8+o(1)}\]
by (\ref{eq:tailbounds}) and a union bound in $i$.

For the second term, we note that if $\tau\leq \log n$ but $S_\tau> (\log n+1)^2$, there must be $0 \le i \le \log n$ such that $X_i\geq \log n$. By (\ref{eq:tailbounds}) and a union bound, the probability of this event without the conditioning on $|\rt|=n$ is $n^{-\omega(1)}$, and the conditioning does not change this because $\p{|\rt|=n}\asymp n^{-3/2}$ by (\ref{eq:treesizeprob}). 
 \end{proof}

We close the section by proving tail bounds on the number of high-degree vertices in $\rt_n^{\bullet}$. Our tool for doing so is the following 
distributional identity for the numbers of children of vertices in random rooted labeled trees, which is an immediate consequence of Lemmas~\ref{lem:pgw0} and~\ref{lem:gw_bfs}. 
 \begin{cor}\label{childpoisson}
Let $\rt_n^{\bullet}$ be a uniformly random rooted tree with vertices labeled by $[n]$, and for $1 \le i \le n$ let $c_i$ be the number of children of $i$ in $\rt_n^{\bullet}$. Then $(c_1,\ldots,c_n)$ has the same distribution as $(P_1,\ldots,P_n)$, where $P_1,\ldots,P_n$ are independent Poisson$(1)$ random variables conditioned to satisfy $\sum_{i=1}^n P_i=n-1$. 
\end{cor}
\begin{cor}\label{cor:deglower}
    Let $\rt_n^{\bullet}$ be a uniformly random rooted tree with vertices labeled by $[n]$. Fix any $k \in \N$, and let 
    $D_{\ge k}(n)$ be the number of vertices of $\rt_n^\bullet$ with at least $k$ children. Then for $t > 0$
    \[
    \p{D_{\ge k}(n) \le \e D_{\ge k}(n)-t} \le \exp(-t^2/(2\e D_{\ge k}(n))),
    \]
    and uniformly over $1 \le k \le \lceil\log n\rceil$, 
    \[
    \e D_{\ge k}(n) \ge (1-o(1))\frac{n}{k!}. 
    \]
\end{cor}
\begin{proof}
The vector $(c_1,\ldots,c_n)$ of numbers of children of the vertices of $\rt_n^\bullet$ is multinomially distributed by Corollary~\ref{childpoisson}, so its entries are negatively associated (see \cite{dubhashi1998balls}). The lower tail bound then follows from \cite[Proposition 5]{dubhashi1998balls} and standard binomial tail estimates. For the expectation bound, let $(X_i,i \ge 1)$ be independent Poisson$(1)$ random variables; by Corollary~\ref{childpoisson}, 
\begin{align*}
\p{c_1=k} & =\p{X_1=k|X_1+\ldots+X_n=n-1}\\
& =\p{X_1=k}\frac{\p{X_2+\ldots+X_n=n-1-k}}{\p{X_1+\ldots+X_n=n-1}}\, \\
& =\frac{1}{k!}\frac{\p{X_2+\ldots+X_n=n-1-k}}{\p{X_1+\ldots+X_n=n-1}}\,.
\end{align*}
Since 
\[
\e D_{\ge k}(n) = n \p{c_1 \ge k} > n \p{c_1=k}\, ,
\]
the result follows by the local central limit theorem. 
\end{proof}
\begin{cor}\label{cor:deltaupper}
Let $\rt_n^{\bullet}$ be a uniformly random rooted tree with vertices labeled by $[n]$, and let $a^*=a^*(n)=\max(m\in \N: m! \le n)$. Then $\e\sqrt{\Delta_{\rt_n^\bullet}} \le a^*(n)+o(1)$ as $n \to \infty$.
\end{cor}
\begin{proof}
Note that $a^*=(1+o(1))(\log n)/(\log \log n)$. Recycling the computation from the previous lemma, uniformly over $a^* \le k \le \lceil \log n\rceil$ we have 
\begin{align*}
\p{c_1=k} 
& = (1+o(1))\frac{1}{k!} \le (1+o(1))\frac{1}{n} \pran{\frac{\log\log n}{\log n}}^{k-a^*-1}\, ,
\end{align*}
so writing $D_{k}(n)$ for the number of vertices of $\rt_n^\bullet$ with exactly $k$ children, we have 
\[
\e{D_k(n)} \le (1+o(1))\pran{\frac{\log\log n}{\log n}}^{k-a^*-1}\, ,
\]
uniformly over $k$ in this range. 

Since $\Delta_{\rt_n^\bullet}$ is at most one greater than $\max(k:D_k(n)>0)$, it follows that 
\begin{align*}
\e\sqrt{\Delta_{\rt_n^\bullet}}
& \le \sqrt{a^*+3} +
\sqrt{1+\log n}\p{a^*+3 \le \Delta_{\rt_n^\bullet} < \log n} + \sqrt{n}\p{\Delta_{\rt_n^\bullet}\ge \log n}\, .
\end{align*}
To bound the second term, we use that 
\begin{align*}
\p{a^*+3 \le \Delta_{\rt_n^\bullet} < \log n}
& \le 
\sum_{k=a^*+2}^{\lceil \log n\rceil -1}\p{D_k(n) > 0}\\
& \le \sum_{k=a^*+2}^{\lceil \log n\rceil -1} \e D_k(n) \le (1+o(1)) \pran{\frac{\log\log n}{\log n}}\, ,
\end{align*}
so 
\[
\sqrt{1+\log n}\p{a^*+3 \le \Delta_{\rt_n^\bullet} < \log n}
= o(1)\, .
\]
To bound the third term, we use Corollary~\ref{childpoisson} to deduce that 
\[
\p{c_1 \ge \lceil \log n\rceil} \le 
\frac{\p{X_1 \ge \lceil \log n\rceil}}{\p{X_1+\ldots+X_{n-1} = n-1}}\, ,
\]
where, like above, $X_1,\ldots,X_n$ are independent and Poisson$(1)$-distributed.
The numerator on the right-hand side is $\tfrac{1+o(1)}{(\lceil \log n \rceil)!}$, which decays faster than polynomially, and the denominator is $\Theta(n^{-1/2})$; so
\[
\sqrt{n}\p{\Delta_{\rt_n^\bullet} \ge \log n} 
\le n^{3/2} \p{c_1 \ge \lceil \log n} = o(1)\, .
\]
Combining the above bounds, we conclude that 
\[
\e\sqrt{\Delta_{\rt_n^\bullet}} \le \sqrt{a^*+3}+o(1)=\sqrt{a^*}+o(1)\, ,
\]
the last identity holding since $a^*=a^*(n) \to \infty$ as $n \to \infty$. 
The result follows since $\Delta_{\rt_n}$ and $\Delta_{\rt_n^\bullet}$ have the same distribution.
\end{proof}

\section{\bf Tree surgery}
\label{sec:treesurgery}

At this point, we already have some information about the structure of the uniformly random tree $\rt_n$. In particular, we can show that $\rt_n$ usually satisfies the properties ${\bf N1-N3}$ described in the proof overview, in Section~\ref{sub:overview}. Recall from that section that for an $n$-vertex tree $T=(V,E)$ and $v \in V$, the cluster $C_T(v)$ is the set of all vertices $w\in V$ such that the unique path from $v$ to $w$ does not contain two consecutive vertices with degree $<\low$. 
We note right away the following fact.

\begin{fact}[Proof omitted]\label{fact:clusters} 
``$w\in C_T(v)$'' is an equivalence relation over $V$, so that $V$ is a disjoint union of clusters. Moreover, each cluster is a connected subset of $T$.\end{fact}

\begin{definition}[Typical trees]\label{def:typical} A $n$-vertex tree $T=(V,E)$ is said to be {\em typical} if 
\begin{enumerate}
\item [{\bf N1}] The maximum degree of $T$ is at least $\nearmax$. 

\item [{\bf N2}] For all $x\in V$, $\sum_{y\in C_T(x)}\deg_T(y)\I{\deg_T(y) \ge \low }\leq 3 \log n$.
\item [{\bf N3}] There are no adjacent nodes that both have degree $\geq 0.9\log n/\log \log n$.
\end{enumerate}
\end{definition}

We show in Section~\ref{sub:basictypical} that this terminology makes sense --- a uniformly random tree is typical with high probability --- and is good enough to control second neighborhood sizes. However, typicality by itself does not suffice for our path-counting arguments to work. A stronger notion is needed for this purpose.  
 \begin{definition}[Nice trees] \label{def:nice} A $n$-vertex tree $\tilde{T}=(\tilde{V},\tilde{E})$ is said to be {\em nice}  if it is typical and satisfies the following additional property: 
\begin{enumerate}\item [{\bf N4}] For all $x\in \tilde{V}$, $C_{\tilde{T}}(x)$ contains less than $9\,\log \log n$ vertices of degree $\geq \low$.
\end{enumerate}
\end{definition}

Thus a nice tree is a typical tree where there are very few nodes of (relatively) high degree in any cluster. The main result of this section is that any upper bound on the top eigenvalue of all nice trees extends to all typical trees.
The proof of the next lemma is given in Section~\ref{sub:proof:typicaltonice}.

\begin{lem}\label{lem:typicaltonice}  {For $n >10$}, for any typical $n$-vertex tree $T=(V,E)$, there exists a nice $n$-vertex tree $\tilde{T}$ with $\lambda_1(T)\leq \lambda_1(\tilde{T})$ and {with $\Delta_{\tilde{T}}=\Delta_T$}.\end{lem}

To prove Lemma \ref{lem:typicaltonice} we use the rewiring result, Lemma~\ref{lem:rewiring}, from Section~\ref{sub:overview}, to perform a sequence of rewirings which can only increase the top eigenvalue. Moreover, Section~\ref{sec:rewiringstructure} establishes that ``good rewirings'' preserve the typicality and and maximum degree of the tree, while decreasing the number of high-degree vertices in a cluster. Lemma \ref{lem:typicaltonice} will follow from iterating ``good rewirings'' until they are no longer possible, which will be shown to imply that the resulting tree is nice.

\subsection{Basic facts on typical trees}\label{sub:basictypical}

We prove here two facts pertaining to typical trees. The first one is that uniformly random trees are usually typical. 

\begin{lem}\label{lem:typical}The uniformly random tree $\rt_n$ is typical (i.e., satisfies Definition \ref{def:typical}) with probability $1-n^{-0.8+o(1)}$.\end{lem}
\begin{proof}It suffices to show that $\rt_n$ satisfies each property 
{\bf N1}-{\bf N3} with probability $\geq 1-n^{-0.8+o(1)}$.

For property {\bf N1}, it follows from results of Moon \cite[Theorem 2 and Lemma 4]{moon1968} that $\Delta_{\rt_n}\geq \nearmax$ with probability $1-e^{-(1+o(1))n^{0.01}}$. 

For property {\bf N3}, the desired probability bound is given by Lemma \ref{lem:adjacent}.

Property {\bf N2} can be dealt with via Corollary \ref{cor:typicalclusterdegrees}. For each $u\in [n]$, let 
\[
K^*(u):=\min(k \in \N: \mathrm{deg}(v) < \log^{1/5}n \mbox{ for all } v \in N_u(2k) \cup N_u(2k+1))
\]
and
\[
Z(u) = \{v \in \rt_n: \mathrm{dist}(u,v)< 2K^*(u), \mathrm{deg}(v) \ge \log^{1/5} n\}
\]
be as in the corollary. By definition of $K^*(u)$, all vertices in $C_{\rt_n}(u)$ lie within distance less than $2K^*(u)$ from $u$. As a consequence, all nodes in $v\in C_{\rt_n}(u)$ with degree $\det_{\rt_n}(v)\geq \low$ lie in the set $Z(u)$. It follows that 
\[\max_{u\in[n]}\sum_{v\in C_{\rt_n}(u)}\deg_{\rt_n}(v)\I{\deg_{\rt_n}(v) \ge \low }\leq \max_{u\in[n]}\sum_{v\in Z(u)}\deg_{\rt_n}(u).\]
Corollary \ref{cor:typicalclusterdegrees} implies that the right-hand side is bounded by $3\log n$, with probability at least $1 - Cn^{-3/2}$. Therefore, {\bf N2} holds with at least this probability.
\end{proof}

Our second result on typical trees is that their second neighborhoods can never be two large.  

\begin{prop}\label{prop:secondneighborhood}In a typical tree with maximum degree $\Delta_T\geq \nearmax$, the second neighborhood of any vertex has size at most \[L\Delta_T^{6/5} (\log\Delta_T)^{1/5},\]
where $L>0$ is a universal constant. 
\end{prop}
\begin{proof}
Fix a vertex $x\in V(T)$. All neighbors of $x$ with degree at least $\low$ lie in $C_{T}(x)$. Therefore, property {\bf N2} implies that the total contribution of these neighbors to the second neighborhood of $x$ is at most $3\log n=O(\Delta_T\log \Delta_T)$. On the other hand, $x$ has at most $\Delta_T$ neighbors of degree less than $\low$ (by {\bf N1}), and each of these contribute at most $\low = O((\Delta_T\log \Delta_T)^{1/5})$ to the second neighborhood. Therefore, there are at most 
\[O\left(\Delta_T\log \Delta_T + \Delta_T^{6/5}\log^{1/5}\Delta_T\right) = O\left( \Delta_T^{6/5}\log^{1/5}\Delta_T\right)\]
second neighbors to $x$.\end{proof}

\subsection{Safe and good rewirings}\label{sec:rewiringstructure}

The idea of the proof of Lemma \ref{lem:typicaltonice} is to perform a sequence of rewirings from a typical tree until it becomes nice.  However, we have to ensure that the tree remains typical and preserves its maximum degree along the way. This motivates the next definition. 

We fix a $n$-vertex tree $T=(V,E)$ for the remainder of the subsection. 

\begin{definition}[Safe/good to rewire] \label{def:safetorewire} A pair $(v,w)\in V^2$ is {\em safe to rewire} for the tree $T$ if $v\neq w$ and the following three conditions are satisfied
\begin{enumerate}
\item[{\bf S1}] $\min\{\mathrm{deg}_T(v),\mathrm{deg}_T(w)\}\geq \low$;
\item[{\bf S2}] $\mathrm{deg}_T(v)+\mathrm{deg}_T(w)\leq \med$;
\item[{\bf S3}] $v$ and $w$ belong to the same cluster in $T$.\end{enumerate}
If moreover $\lambda_1(T_{v,w})\geq \lambda_1(T)$, $(v,w)$ is said to be {\em good to rewire} for $T$.\end{definition}

The following is an immediate consequence of Lemma \ref{lem:rewiring} and the fact that the definition of ``safe to rewire'' is symmetric in $v$ and $w$.

\begin{fact}\label{fact:safe}If $(v,w)\in V^2$ is safe to rewire for $T$, then either $(v,w)$ or $(w,v)$ is good to rewire. \end{fact}

It turns out that rewiring a pair that it is safe to rewire has a constrained effect on the cluster structure of $T$. This is the content of the following Lemma. 
\begin{lem}\label{lem:clusterrefines}Let $T$ be a tree and $(v,w)$ be safe to rewire. Then any cluster of $T_{v,w}$ is fully contained in a cluster of $T$.\end{lem}

\begin{proof} By Fact \ref{fact:clusters}, we must show that any two vertices $a,b\in V$ that belong to the same cluster in $T_{v,w}$ also belong to the same cluster in $T$. To do this, fix $a,b\in V$ in the same cluster in $T_{v,w}$, and consider the unique path $Y=y_0y_1\dots y_m$ connecting $a=y_0$ to $b=y_m$ in $T_{v,w}$. The fact that $a,b$ lie in the same cluster of $T_{v,w}$ is equivalent to the statement that
\begin{equation}\label{eq:notwoconsecutivelows}\max\{\deg_{T_{v,w}}(y_{i-1}),\deg_{T_{v,w}}(y_{i})\}\geq \low\mbox{ for each }i\in [m].\end{equation} 
Recall from Section \ref{sub:overview} that $T_{v,w}$ is defined by letting $u$ be the unique neighbor of $v$ in $T$ along the $vw$ path, and replacing each edge $sv$ with $s\in N_T(v)-\{u\}$ with an edge $sw$. Fixing this vertex $u$, we see that there are two possibilities. \\

\noindent{\bf Case 1: $Y$ does not contain $w$.} In this case, $Y$ is also the path in $T$ connecting $a,b$. Therefore, $a$ and $b$ lie in the same cluster in $T$ if there are no two consecutive vertices in $Y$ have degree $<\low$ in $T$. To see that this is indeed true, notice that all vertices in $V-\{v,w\}$ have the same degree in $T$ as in $T_{v,w}$. Moreover, {\bf S1} ensures that $\mathrm{deg}_T(v)\geq \low$. Therefore, condition (\ref{eq:notwoconsecutivelows}) for $Y$ as a path $T_{v,w}$ also applies to $Y$ as a path in $T$, and we are done. \\

\noindent{\bf Case 2: $Y$ contains $w$.} In this case, it suffices to argue that $a,b\in C_T(w)$. In fact, we only prove $a\in C_T(w)$, as $b\in C_T(w)$ follows from ``reversing the path."

Since $Y$ is a path in $T_{v,w}$, there exists exactly one index $j\in [m]\cup \{0\}$ with $y_j=w$. We further split into cases.

\begin{itemize}
    \item {\bf Case 2.1: $j=0$.} Then $a=w$ and we are done.
    \item{\bf Case 2.2: $j>0$.} Then $y_0\dots y_{j-1}$ is a path in $T_{v,w}$ not containing $w$, so  $a=y_0$ and $y_{j-1}$ lie in the same cluster of $T$ by Case 1. It now suffices to show that $y_{j-1}\in C_T(w)$, which again requires a considering two cases.
    \begin{itemize}
\item{\bf Case 2.2.1: $y_{j-1}$ is a neighbor of $w$ in $T$.} Then $y_{j-1}\in C_T(w)$ because {\bf S1} guarantees $\deg_T(w)\geq \low$;
\item {\bf Case 2.2.2: $y_{j-1}$ is not a neighbor of $w$ in $T$.} It is still true that $y_{j-1}w$ is an edge in $T_{v,w}$, so it must be that $y_{j-1} \in N_T(v)-\{u\}$, with $u$ as above. In this case, {\bf S3} guarantees that $v\in C_T(w)$, and {\bf S1} implies that $v$ has degree $\geq \low$, so $y_{j-1}\in C_T(v)=C_T(w)$, as desired.\end{itemize}
\end{itemize}

\end{proof}

We now argue that the performing rewirings on safe-to-rewire vertices preserves the typicality and maximum degrees of a tree.

\begin{lem}\label{lem:safetorewire}Assume $T$ is typical and $(v,w)\in V^2$ is safe to rewire for $T$. Then $T_{v,w}$ and $T_{w,v}$ both have the same vertex set and the same maximum degree as $T$. Moreover, $T_{v,w}$ and $T_{w,v}$ are both typical.\end{lem}

\begin{proof} We only consider $T_{v,w}$ without loss of generality. Notice that the vertex set is preserved by construction. As for the maximum degree, we have $\Delta_T\geq \nearmax$ since $T$ is typical, and {\bf S2} means that neither $v$ or $w$ achieve the maximum degree. For $T_{v,w}$, we have
\begin{equation}\label{eq:bothdegrees}\mathrm{deg}_{T_{v,w}}(v)=1,\, \mathrm{deg}_{T_{v,w}}(w)=\mathrm{deg}_{T}(v)+\mathrm{deg}_{T}(w)-1<\med,\end{equation}
and all vertices $x\in V\setminus\{v,w\}$ have the same degree as in $T$. It follows that $\Delta_{T_{v,w}}=\Delta_T$ is still achieved by some vertex that is neither $v$ nor $w$. 

We now argue that $T_{v,w}$ is typical, as per Definition \ref{def:typical}. Property {\bf N1} is automatic because $T$ and $T_{v,w}$ have the same maximal degree. 

To prove property {\bf N2} for $T_{v,w}$, first note that Lemma \ref{lem:clusterrefines} implies $C_{T_{v,w}}(x)\subset C_T(x)$ for all $x\in V$. Since $T$ is typical (and in particular satisfies {\bf N2}), we obtain
\[\sum_{y\in C_{T}(x)}\deg_{T}(y)\I{\deg_{T}(y) \ge \low}\leq 3\log n.\]
Therefore, it suffices to show that
\begin{equation}\label{eq:claimT2}\sum_{y\in C_{T}(x)}\deg_{T_{v,w}}(y)\I{\deg_{T_{v,w}}(y) \ge \low}\leq \sum_{y\in C_{T}(x)}\deg_{T}(y)\I{\deg_{T}(y) \ge \low},\end{equation}
where $x\in V$ is arbitrary. This property is automatic if $v,w\not\in C_T(x)$, because in this case the degrees in the left-hand side of (\ref{eq:claimT2}) are the same as in the right-hand side. If on the other hand $v\in C_T(x)$ or $w\in C_T(x)$, it must be that $v,w\in C_T(x)$, because {\bf S3} guarantees that $v,w$ lie in the same cluster of $T$. Property {\bf S1} and equation (\ref{eq:bothdegrees}) then imply that the left-hand side of (\ref{eq:claimT2}) is strictly smaller than the right-hand side: the combined contribution of the terms $y=u$ and $y=v$ decreases by $1$, and all other terms remain unchanged.

Finally, equation (\ref{eq:bothdegrees}) implies that no nodes of degree $\geq\med$ are ever created by the rewiring operation. Since $T$ satisfies ${\bf N3}$, it follows that property ${\bf N3}$ also holds for $T_{v,w}$. 
\end{proof}

\subsection{From typical to nice trees}\label{sub:proof:typicaltonice}

We now have all the tools we need to prove Lemma \ref{lem:typicaltonice}.

\begin{proof}[Proof of Lemma \ref{lem:typicaltonice}]Consider the following procedure applied to $T$.

\begin{enumerate}
\item Set $\tilde{T}\leftarrow T$;
\item While there exists a good-to-rewire pair $(v,w)$ for $\tilde{T}$, set $\tilde{T}\leftarrow T_{v,w}$
$\tilde{T}\leftarrow \tilde{T}_{v,w}$.
\item Output $\tilde{T}$.
\end{enumerate}
Note that the vertex set of $\tilde{T}$ is always $V$; only the edge set changes. The above procedure must terminate, since each rewiring operation increases the number of leaves in the tree (cf. (\ref{eq:sendedges}) {and the fact that $v$ is not itself a leaf by condition~{\bf S1} in Definition~\ref{def:safetorewire}}).
By Lemma \ref{lem:safetorewire}, the final tree $\tilde{T}$ is typical, and has the same vertex set and the same maximum degree as $T$. We also know that $\lambda_1(\tilde{T})\geq \lambda_1(T)$ and that $\tilde{T}$ cannot contain any pair $(v,w)\in V^2$ that is good to rewire. Therefore, by Fact \ref{fact:safe}, it also cannot contain any pair that is safe to rewire (otherwise the procedure could continue). It follows that $\tilde{T}$ has the following property:

\begin{enumerate}
    \item[{\bf \~{N}4}] For any two vertices $v,w\in V$ that lie in the same cluster in $\tilde{T}$,
    \[\min\{\mathrm{deg}_{\tilde{T}}(v),\mathrm{deg}_{\tilde{T}}(w)\}\geq \low\Rightarrow\mathrm{deg}_{\tilde{T}}(v)+\mathrm{deg}_{\tilde{T}}(w)>\med.\hfill\]
\end{enumerate}

To finish, we show the following claim.
\begin{claim}If $\tilde{T}$ is typical and satisfies {\bf \~{N}4}, then for all $x\in V$, $C_{\tilde{T}}(x)$ contains < $9\,\log \log n$ vertices of degree $\geq \low$; i.e. $\tilde{T}$ is nice.\end{claim}
To see this, fix any $x\in V$, and consider the sum
\[\sum_{y\in C_{\tilde{T}}(x)}\deg_{\tilde{T}}(y)\I{\deg_{\tilde{T}}(y) \ge \low}.
\]
On the one hand, this sum is bounded from above by $3\log n$ because ${\tilde{T}}$ is typical (cf. property {\bf N2} in Definition \ref{def:typical}). On the other hand, {\bf \~{N}4} implies that, for any two vertices $v,w\in C_{\tilde{T}}(x)$ with $\min\{\mathrm{deg}_{\tilde{T}}(v),\mathrm{deg}_{\tilde{T}}(w)\}\geq \low$, their joint contribution to the sum is at least {$\med$}. Therefore
\[\med\,\left\lfloor \frac{\#\{v\in C_{\tilde{T}}(x)\,:\, \mathrm{deg}_{\tilde{T}}(v)\geq \low\}}{2}\right\rfloor\leq 3\log n,\]
from which the claim follows after some simple estimates (this is where the condition that $n > 10$ is required).
\qedhere
\end{proof}

\section{\bf Putting it all together}
\label{sec:endofproof}

In this section we finish the proof of the main result. We will use eigenvalue bound from Corollary~\ref{cor:countpathsIJ}, and we encourage the reader to refresh themselves on the statement of that corollary (and that of Proposition~\ref{prop:countpathsIJ}, on which the corollary relies). 

The main missing technical ingredient is the construction of a good enough $(M,H)$-code for nice trees. The code we shall use appears in Figure~\ref{fig:rightcode}. (Like in Section~\ref{sec:refined}, in Figure~\ref{fig:rightcode} we use the $*$ symbol as a wildcard, meaning that any of $\ell,m,h$ could be there.)
\begin{figure}[hbt]
    \centering
    \includegraphics[width=0.8\linewidth]{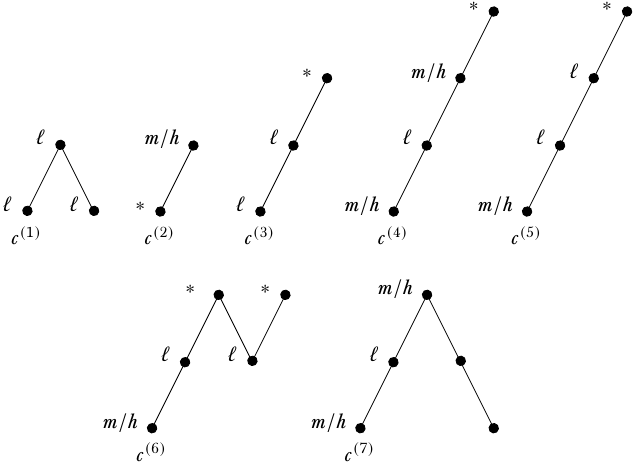}
    \caption{The code we use in our analysis. Each of $c^{(2)},\ldots,c^{(7)}$ in the above figure corresponds to multiple code words, as there are two choices at each location marked $m/h$ and three choices at each location marked $*$. It follows that the total number of words in the above code is $1+6+3+12+6+18+4=50$.}
    \label{fig:rightcode}
\end{figure}
The following claim asserts that it is indeed good enough for our purposes.
\begin{claim}\label{cla:final}For the $(M,H)$-code $C$ in Figure \ref{fig:rightcode}, the parameter $\eps$ in Proposition \ref{prop:countpathsIJ} satisfies:
\[\eps\leq L\,\left(\frac{\log^{2/15}\Delta_T}{\Delta_T^{8/15}}\right).\]
for any nice tree $T$ on $n\ge 3$ vertices (cf. Definition \ref{def:nice}) with maximum degree $\Delta_T$. Here, $L>0$ is a universal constant.
\end{claim}

\begin{proof} We assume that the tree $T$ has maximum degree $\Delta_T$ and is nice, as per Definition \ref{def:nice}. To recapitulate, this means that 
\begin{enumerate}
\item [{\bf N1}] The maximum degree of $T$ is $\Delta_T\geq \nearmax$. 
\item [{\bf N2}] For all $x\in V$, $\sum_{y\in }\deg_{T}(y)\I{\deg_{T}(y) \ge \low}\leq 3 \log n$.
\item [{\bf N3}] There are no adjacent nodes that both have degree $\geq 0.9\log n/\log \log n$.
\item [{\bf N4}] For all $x\in V$, $C_{T}(x)$ contains  $<9\,\log \log n$ vertices of degree $\geq \low$.
\end{enumerate}

To apply our combinatorial arguments, we let $H$ denote the set of vertices $z\in V$ with degrees $\deg_T(z)\geq \low$. 

For the code $C$ shown in Figure \ref{fig:rightcode}, our task is to show that
\[\eps^{(b)}:= \frac{\Delta_T(c^{(b)})^{\frac{2}{(\len(c^{(b)})+\f(c^{(b)}))}}}{\Delta_T}\leq L\frac{\log^{2/15} \Delta_T}{\Delta_T^{8/15}}\] 
for each index $b \le 7$. Here, we
recall that $\Delta_T(c^{(b)})$ is the maximum number of ways one can insert $c^{(b)}$ in $T$ (starting from a vertex of the appropriate kind); $\f(c^{(b)})$ is the final value of the codeword; and $\len(c^{(b)})$ is the length of the codeword. The parameter $\eps$ in Claim~\ref{cla:final} is (up to a universal constant $(2a)^2=100^2$) the maximum of the $\eps^{(b)}$. (Above we are slightly abusing notation since there are actually 50 codewords, as noted in the caption of Figure~\ref{fig:rightcode}, but this should not cause confusion in the sequel; for each $b \le 7$, the bound we prove  holds for any of the codewords corresponding to $c^{(b)}$.)

Note that, by condition {\bf N1}, $\log n\leq L\Delta \log \Delta$ and $\log \log n\leq L\log \Delta$. We will use these bounds repeatedly in what follows.

\noindent {\bf Codeword $c^{(1)}$.} Here $\f(c^{(1)})=0$ and $\len(c^{(1)})=1$. Also, $\Delta_T(c^{(1)})\leq \low\leq  L\Delta_T^{1/5} \log^{1/5} \Delta_T$. We conclude that:
\[\eps^{(1)}\leq L\frac{\log^{2/5} \Delta_T}{\Delta_T^{3/5}}.\]

\noindent {\bf Codeword $c^{(2)}$.} In this case $\f(c^{(2)})=\len(c^{(2)})=1$. To compute $\Delta_T(c^{(2)})$, let $x$ be the first vertex in the path. Then the second vertex $y$ belongs to $C_T(x)$ because $y$ has degree $\geq \low$. Condition {\bf N4} means that, given $x$, there are at most $9\log \log n\leq L\log \Delta_T$ choices for $y$. We conclude that 
\[\eps^{(2)}\leq \frac{L\log \Delta_T}{\Delta_T}.\]

\noindent {\bf Codeword $c^{(3)}$.} Final value and length are given by $\f(c^{(3)})=2$ and $\len(c^{(3)})=2$. Moreover, given the initial vertex $x$, there are $(\low)^2\leq L\Delta_T^{2/5}\log^{2/5}\Delta_T$ choices for vertices $y\sim x$ and $z\sim y$. We conclude that: 

\[\eps^{(3)}\leq L\frac{\log^{1/5} \Delta_T}{\Delta_T^{4/5}}.\]

\noindent {\bf Codeword $c^{(4)}$.} In this case $\f(c^{(4)})=\len(c^{(4)})=3$. To compute $\Delta_T(c^{(4)})$, let $x,y,z,w$ be the vertices along the path; our goal is to bound the number of choices for $w$. To do this, notice that $z\in C_T(x)$ because both $x$ and $z$ have degree $\geq \low$. Therefore, $w$ is a neighbor of a node $z\in C_T(x)$ with degree $\deg_T(z)\geq \low$. Condition {\bf N2} guarantees that there are at most $3\log n\leq L\Delta_T \log \Delta_T$ choices for $w$. We conclude

\[\eps^{(4)}\leq L\frac{\log^{1/3} \Delta_T}{\Delta_T^{2/3}}.\]

\noindent {\bf Codeword $c^{(5)}$.} In this case again $\f(c^{(5)})=\len(c^{(5)})=3$. We again let $x,y,z,w$ be the vertices where the codeword is placed, in order they appear. The number of choices for $y$ given $x$ is at most $\Delta_T$, and one can then choose $z$ and $w$ in $(\low)^2\leq L(\Delta_T \log \Delta_T)^{2/5}$ ways. We conclude that $\Delta_T(c^{(5)})\leq L\,\Delta_T^{7/5}\log^{2/5}\Delta_T$ and 
\[\eps^{(5)}\leq L\frac{\log^{2/15} \Delta_T}{\Delta_T^{8/15}}.\]

\noindent{\bf Codeword $c^{(6)}$.} In this case, $\f(c^{(6)})=2$ and $\len(c^{(6)})=4$. To count $\Delta_T(c^{(6)})$, notice that this codeword corresponds to a walk visiting vertices $x,y,z$, returning to $y$ and then visiting some vertex $w$. Given $x$, one can choose $y$ in at most $\Delta_T$ ways; having chosen $y$, one can choose the pair $z,w$ of neighbours of $y$ in at most $(\low)^2\leq L(\Delta_T\log \Delta_T)^{2/5}$ ways since $y$ must have low degree. Therefore, $\Delta_T(c^{(6)})\leq L\,\Delta_T^{7/5}\log^{2/5}\Delta_T$ and 
\[
\eps^{(6 )}\leq L\frac{\log^{2/15} \Delta_T}{\Delta_T^{8/15}}.
\]

\noindent {\bf Codeword $c^{(7)}$.} In this final case we have $\f(c^{(7)})=2$ and $\len(c^{(7)})=4$.
This codeword corresponds to a walk of the form $xyzyx$ where both $x$ and $z$ have degree at least $(\log n)^{1/5}$. 
But then $x$ and $z$ belong to same cluster and as in the case of codeword $c^{(2)}$, we can argue that by  property {\bf N4},  that there are at most $9\log \log n \le L \log \Delta_T$ choices for $z$. 
Thus $\Delta_T(c^{(7)})\le L\log \Delta_T$
and so 
\[
\eps^{(7)}\leq L\frac{\log^{1/3} \Delta_T}{\Delta_T}.\qedhere
\]
\end{proof}

We now turn to the proofs of the main theorems.

\begin{proof}[Proof of 
Theorem~\ref{thm:main}] 
We first use the probability bound of the theorem to prove the expectation bound, then prove the probability bound.

Since $0 \le \lambda_1(\rt_n) \le 2\sqrt{\Delta_{\rt_n}}\le 2\sqrt{n-1}$, if $K \in \N$ is large enough that $(\log^2 x)/\sqrt{x}$ is decreasing for $x \ge K$ then we have 
\begin{align*}
\e|\lambda_1(\rt_n)-\sqrt{\Delta_{\rt_n}}|
& \le 
\sqrt{K}\p{\Delta_{\rt_n} \le K}
\\
& + L\left(\frac{\log^2 K}{\sqrt{K}}\right)^{1/15}
\p{\Delta_{\rt_n} \ge K,|\lambda_1(\rt_n)-\sqrt{\Delta_{\rt_n}}|\le L\left(\frac{\log^2 \Delta_{\rt_n}}{\sqrt{\Delta_{\rt_n}}}\right)^{1/15}}
\\
& +\sqrt{2\log n} 
\p{\Delta_{\rt_n}< 2\log n,|\lambda_1(\rt_n)-\sqrt{\Delta_{\rt_n}}|>L\left(\frac{\log^2 \Delta_{\rt_n}}{\sqrt{\Delta_{\rt_n}}}\right)^{1/15}}\\
& + 2\sqrt{n-1}\p{\Delta_{\rt_n}\ge 2 \log n}.
\end{align*}
The first term of this upper bound tends to $0$ since $\Delta_{\rt_n} \to \infty$ in probability as $n \to \infty$; this follows from the fact, noted in the introduction, that $\Delta_{\rt_n}$ is approximately distributed as the maximum of $n$ independent Poisson$(1)$ random variables. The probability in the second term on the right is at most $1$. The third term on the right tends to $0$ by the probability bound of the theorem. The fourth term on the right tends to $0$ by Fact~\ref{fact:maxdeg}. We deduce that 
\[
\e|\lambda_1(\rt_n)-\sqrt{\Delta_{\rt_n}}| \le L\left(\frac{\log^2 K}{\sqrt{K}}\right)^{1/15} + o(1)\, .
\]
Since we can take $K \in \N$ arbitrarily large, it follows that $\e|\lambda_1(\rt_n)-\sqrt{\Delta_{\rt_n}}| \to 0$.

We now turn to proving the probability bound of the theorem. 
From Lemma \ref{lem:typical} we know that the random labelled tree $\rt_n$ is typical (Definition \ref{def:typical}) with probability at least $1-n^{-0.8+o(1)}$. Therefore, it suffices to show the following deterministic statement: for any typical tree $T$ on $n$ vertices and maximum degree $\Delta_T$,
\begin{equation}\label{eq:upperdegree}\sqrt{\Delta_T}\leq \lambda_1(T)\leq \sqrt{\Delta_T} + L\,\frac{\log^{2/15}\Delta_T}{\Delta_T^{1/30}},\end{equation}
with a universal constant $L>0$. The lower bound is true for all graphs with maximal degree $\Delta_T$, so it suffices to prove the upper bound. By adjusting the value of $L$, we may also assume that $n$ is large. 
In the remainder of this section, $L$ always denotes a universal constant (not depending on $n$) whose value may change line to line.

Applying Lemma \ref{lem:typicaltonice}, we obtain a nice tree $\tilde{T}$ on $n$ vertices with the same maximum degree \[\Delta_T\geq \frac{0.99\log n}{\log \log n}\] as $T$ and with $\lambda_1(T)\leq \lambda_1(\tilde{T})$.

Properties {\bf N1} and {\bf N3}, together with the preceding bound on $\Delta_T$, imply that $\tilde{T}$ does not have adjacent nodes of degree $\geq 0.95\Delta_T$. Corollary \ref{cor:countpathsIJ} applies to the tree $\tilde{T}$ with the $(M,H)$-code $C$ in Figure \ref{fig:rightcode}, which has $a=50$ codewords.
It follows that
\begin{align*}\lambda_1(\tilde{T})& =
\sqrt{\Delta_T} \left(1 + \frac{\Delta_T^{(2)}}{\Delta_T^2}\right)\,\sqrt{1+e^{20}a^2\eps}.\end{align*}

Using the value of $\eps$ given by Claim~\ref{cla:final}, we obtain that  
\[\lambda_1(\tilde{T})\leq \sqrt{\Delta_T} \left(1 + \frac{\Delta_T^{(2)}}{\Delta_T^2}\right)\,\left(1+L\,\frac{\log^{2/15}\Delta_T}{\Delta_T^{8/15}}\right). \]
Proposition \ref{prop:secondneighborhood} gives us
\[\Delta_T^{(2)}\leq L\Delta_T^{6/5} (\log\Delta_T)^{1/5},\]
and the upper bound in (\ref{eq:upperdegree}) now follows after some simple estimates. 
\end{proof}

For the second theorem, we require a final lemma, relating eigenvalues and the degrees of nodes in trees.
\begin{lem}\label{lem:tree_ev_bd}
    Let $T$ be a tree with $n$ vertices. List the vertices of $T$ in decreasing order of degree as $v_1,\ldots,v_n$, and write $d_i$ for the degree of $v_i$. Then for any integer $k \le n/2$ we have $\lambda_k(T)\ge \sqrt{d_{2k}-1}$. 
\end{lem}
\begin{proof}
Fix $k \le n/2$ and choose a set $S \subset [2k]$ of size $k$ such that $\{v_i,i \in S\}$ is an independent set in $T$; such $S$ exists since $T$ is bipartite. Then fix a root for $T$, and for each $i \in S$ write $C_i$ for the resulting set of children of $v_i$ in $T$. Note that $|C_i|=d_i$ if $v_i$ is the root of $T$ and otherwise $|C_i|=d_i-1$. 

Then for $i \in S$ define a vector $x^{(i)}$ by 
\[
x^{(i)}_v
= \begin{cases}
    \sqrt{|C_i|} & \mbox{if } v = v_i\\
    1            & \mbox{if } v \in C_i\\
    0            &  \mbox{otherwise.}
    \end{cases}
\]
Then $\langle x^{(i)}, A x^{(i)} \rangle \ge \sqrt{|C_i|} |x^{(i)}|^2 \ge \sqrt{d_i-1} |x^{(i)}|^2$.
Moreover, $\langle x^{(i)},x^{(j)}\rangle=0$ for $i \ne j$ with $i,j \in S$, since $\{v_i,i \in S\}$ is an independent set. It follows by Rayleigh's principle that $\lambda_k(T) \ge \min(\sqrt{d_i-1}, i \in S)\ge \sqrt{d_{2k}-1}$. 
\end{proof}

\begin{proof}[Proof of Theorem~\ref{thm:main2}]
Recall from Corollary~\ref{cor:deltaupper} that we define $a^*(n)=\max(m\in \N: m! \le n)$. 
By \cite[Theorem~1]{carrgohschmutz}, 
the median $a_n$ of $\Delta_{\rt_n}$ satisfies that 
\begin{equation}\label{eq:ancompare}
|a_n - a^*| = O(1)\, ,
\end{equation}
so in particular,
\[
a_n =(1+o(1))(\log n)/\log \log n.
\]
It then follows from Corollary~\ref{cor:deltaupper} that 
\[
\e{\Delta_{\rt_n}} \le \sqrt{a^*(n)}+o(1)=\sqrt{a_n}+o(1)\, ,
\]
so, by Theorem~\ref{thm:main} and the triangle inequality, $\e\lambda_{1}(\rt_n) \le \sqrt{a_n}+o(1)$. 
Since $\lambda_1(\rt_n)\ge \lambda_{k(n)}(\rt_n)$,  to prove the theorem it remains to show that $\e\lambda_{k(n)}(\rt_n) \ge \sqrt{a_n}-o(1)$, where $k(n)=\lceil e^{(\log n)^{\beta}}\rceil$ for some $\beta \in (0,1/2)$ fixed. 

Let $\alpha \in (\beta,1/2)$ and take
$a=\lfloor a_n-\log^\alpha n\rfloor$; then it follows from Corollary~\ref{cor:deglower} that 
\[
\e D_{\ge a}(n) \ge (1-o(1))\frac{n}{a!} > 4 e^{(\log n)^\beta}\, ,
\]
the second bound holding for $n$ sufficiently large. Using the lower tail bound from Corollary~\ref{cor:deglower} it then follows that 
\[
\p{D_{\ge a}(n) \le 2e^{(\log n)^\beta}} 
\le \exp(-2 e^{(\log n)^\beta})\, .
\]
Since the degree of a node in $\rt_n$ is at least its number of chidren in $\rt_n^\bullet$, it then follows from Lemma~\ref{lem:tree_ev_bd} that 
\[
\p{\lambda_{\lceil e^{(\log n)^{\beta}}\rceil}(\rt_n) < \sqrt{a-1}} \le \exp(-2 e^{(\log n)^\beta})\, .
\]
Since $\sqrt{a-1} \ge \sqrt{a_n}-o(1)$ and $k(n)=\lceil e^{(\log n)^{\beta}}\rceil$, it follows that 
\[
\e{\lambda_{k(n)}(\rt_n)} \ge \sqrt{a_n}-o(1)\, ,
\]
as required. 
\end{proof}

\addtocontents{toc}{\SkipTocEntry} 
\section{\bf Acknowledgements}
Much of this work took place when LAB and GL were visiting RIO at IMPA; they thank  IMPA for the hospitality. During the final stages of the work, LAB was hosted by Alex Scott at Merton College, and thanks them both for their hospitality.

\small 
\addtocontents{toc}{\SkipTocEntry}
\bibliographystyle{plainnat}
\bibliography{topeigenvalue}
\normalsize
                 
\appendix

\end{document}